\documentclass[a4paper]{article}
\usepackage{lipsum}
\usepackage{amsfonts}
\usepackage{graphicx}
\usepackage{epstopdf}
\usepackage{algorithmic}
\usepackage{relsize}
\usepackage{tikz}

\usepackage{tabularx}
\usepackage{booktabs}

\usetikzlibrary{positioning}
\ifpdf
  \DeclareGraphicsExtensions{.eps,.pdf,.png,.jpg}
\else
  \DeclareGraphicsExtensions{.eps}
\fi

\usepackage{enumitem}
\usepackage{tikz}
\usepackage{tkz-euclide}
\setlist[enumerate]{leftmargin=.5in}
\setlist[itemize]{leftmargin=.5in}

\usepackage{amsmath,amsthm}
\usepackage{hyperref}
\usepackage{algorithm}
\usepackage{geometry}

\providecommand{\keywords}[1]{\textbf{\textit{Keywords: }} #1}
\providecommand{\MSCcodes}[1]{\textbf{\textit{MSC Codes: }} #1}

\newtheorem{theorem}{Theorem}
\newtheorem{proposition}[theorem]{Proposition}
\newtheorem{remark}[theorem]{Remark}
\newtheorem{definition}[theorem]{Definition}
\newtheorem{lemma}[theorem]{Lemma}

\newcommand*\dx{\mathop{}\!\mathrm{d}}

\title{Exact Parameter Identification in PET Pharmacokinetic Modeling Using the Irreversible Two Tissue Compartment Model
\thanks{This work was partly funded by the NIH grants P41-EB017183 and R01-EB031199-02.}}
\author{Martin Holler \thanks{Institute of Mathematics and Scientific Computing, University of Graz. MH further is a member of NAWI Graz (\href{https://www.nawigraz.at}{www.nawigraz.at}) and of BioTechMed Graz (\href{https://biotechmedgraz.at}{biotechmedgraz.at}) (\href{mailto:martin.holler@uni-graz.at}{martin.holler@uni-graz.at})} \and Erion Morina \thanks{Institute of Mathematics and Scientific Computing, University of Graz.
    (\href{mailto:erion.morina@uni-graz.at}{erion.morina@uni-graz.at}).}\and
  Georg Schramm \thanks{Radiological Sciences Laboratory, Stanford University, Stanford, CA, and Department of Imaging and Pathology, KU Leuven, Belgium. (\href{mailto:gschra2@stanford.edu}{gschra2@stanford.edu})}}

\usepackage{amsopn}

\newcommand{\N}{\mathbb{N}}
\newcommand{\R}{\mathbb{R}}

\newcommand{\Kb}{\textbf{K}}

\newcommand{\Mc}{\mathcal{M}}

\newcommand{\Pc}{\mathcal{P}}
\newcommand{\Dc}{\mathcal{D}}

\newcommand{\tis}{{\text{tis}}}
\newcommand{\tot}{{\text{bl}}}
\newcommand{\art}{{\text{art}}}
\newcommand{\fbv}{{\text{fbv}}}
\newcommand{\pet}{{\text{PET}}}

\newcommand{\free}{{\text{fr}}}
\newcommand{\bd}{{\text{bd}}}

\newcommand{\contentskip}[1]{}
\renewcommand{\contentskip}[1]{#1}

\begin{document}

\maketitle

\begin{abstract}
	This work is concerned with the identifiability of metabolic parameters from multi-region measurement data in quantitative PET imaging. It shows that, for the frequently used two-tissue compartment model and under reasonable assumptions, it is possible to uniquely identify metabolic tissue parameters from standard PET measurements, without the need of additional concentration measurements from blood samples. This result, which holds in the idealized, noiseless scenario, indicates that costly concentration measurements from blood samples in quantitative PET imaging can be avoided in principle. The connection to noisy measurement data is made via a consistency result, showing that exact reconstruction is maintained in the vanishing noise limit. Numerical experiments with a regularization approach are further carried out to support these analytic results in an application example.
\end{abstract}


\MSCcodes{
	65L09, 94A12, 92C55, 34C60}

\section{Introduction}

Positron emission tomography (PET) is a non-invasive clinical technique that
images the four dimensional spatio temporal distribution of a radio tracer
in-vivo. In quantitative dynamic PET imaging, after tracer injection, several
3D PET images at different time points are acquired and reconstructed. The
injected tracer is supplied to the tissues via the arteries and capillaries. As
a response to this ``arterial tracer input'', the tracer is exchanged with
tissues. This exchange can include reversible and/or irreversible binding and
eventually metabolization of the tracer.

Using the right tracer, the time series of reconstructed PET images reflecting
the tissue response, a measurement of the arterial tracer input, and a
dedicated pharmacokinetic model, it is possible to generate images reflecting
certain physiological parameters. Depending on the tracer, these parametric
images can reflect e.g. blood flow, blood volume, glucose metabolism or neuro
receptor dynamics.

Pharmacokinetic modeling in PET is commonly performed using compartment models,
where the compartments usually reflect tissue subspaces. Example for such
subspaces are the extra cellular spaces where the tracer is free or
bound.\footnote{Note that the exact interpretation of the biological meaning of
	the compartments is highly non-trivial.} The dynamics between the arterial
blood and tissue compartments is typically described using
ordinary-differential-equation (ODE) models. For PET tracers with irreversible
binding, such as [\textsuperscript{18}F]Fluorodeoxyglucose (FDG)
\cite{sokoloff_mapping_1978} or [\textsuperscript{11}C]Clorgyline
\cite{logan_strategy_2002}, the irreversible two tissue compartment model can
be used to describe the tracer dynamics, see Figure
\ref{fig:comparment_model_scheme} for a scheme of this model.

The identification of the kinetic parameters describing the tissue response to
the arterial tracer supply in a given region is commonly done using the
following input data:
\begin{enumerate}
	\item The tracer concentration in tissue $C_\tis(t)$. This quantity, which is equal
	      to the sum of all the tracer concentrations in all extra-vascular compartments,
	      can be directly obtained from the time series of reconstruction PET images -
	      either on a region-of-interest (ROI) or at voxel level.
	\item The arterial ``input function'' or plasma concentration $C_\art(t)$ of the
	      non-metabolized tracer describing the concentration of tracer in the arterial
	      blood plasma that is supplied to tissue and available for exchange (and
	      metabolism). A direct measurement of $C_\art(t)$ is complicated. Typically, it
	      is based on an external measurement of a time series of arterial blood samples
	      taken from a patient, e.g. using a well counter. Unfortunately, the total
	      arterial blood tracer concentration $C_\tot(t)$ obtained from the well counter
	      measurements usually overestimates $C_\art(t)$ since the measured activity of
	      the blood samples also includes activity from radioactive molecules that are
	      not available for exchange with tissue because of (i) parts of the radio tracer
	      being bound to plasma proteins and (ii) activity originating from metabolized
	      tracer molecules that were transfered back from tissue into blood. Measuring
	      the contributions of the latter two effects, summarized in the parent plasma
	      fraction $f(t) = C_\art(t) / C_\tot(t)$, requires further advanced chemical
	      processing and analysis of the blood samples and is thus time consuming and
	      expensive.
\end{enumerate}
To avoid the necessity of arterial blood sampling, which itself is a very
challenging process in clinical routine, many attempts have been made to derive
the arterial input function directly from the reconstructed PET images (also
called ``image-based arterial input function'') by analyzing the tracer
concentration in regions of interest of the PET images containing arterial
blood, such as the left ventricle, the aorta, or the carotic arteries. Note,
however, that by using any image-based approach for the estimation of the
arterial input function, the contributions of tracer bound to plasma proteins
and metabolized tracers cannot be determined. In other words, any image-based
approach can only estimate $C_\tot(t)$ instead of $C_\art(t)$ since no
measurement of the parent plasma fraction $f(t)$ is available.

Motivated by the problem, we consider the question whether the identification
of kinetic parameters using data from multiple anatomical regions of interest
and the irreversible two tissue compartment model is possible without having
access to the common parent plasma fraction $f(t)$ and/or the common total
arterial blood tracer concentration $C_\tot(t)$.

In existing literature on modeling approaches for quantitative PET, see for
instance \cite{Ver13} and \cite{Ton15} for a review, this question has been
addressed from a computational perspective. The work \cite{Ver13}, for
instance, accounts for a low number of measurements of the arterial
concentration of non-metabolized PET tracer by using a non-linear mixed effect
model for the parent plasma fraction, i.e., the parameters defining the parent
plasma fraction are modeled as being partially patient-specific and partially
the same for a population sample. Moreover, the general idea of jointly
modeling the tissue response in different anatomical regions to obtain unknown
common parameters or to reduce the variance in the estimated region-dependent
parameters has been proposed in \cite{raylman_1994, huesman_1997,
	ogden_2015,chen_2019}.

Despite these and many more computational approaches for parameter
identification in pharmacokinetic modeling using the irreversible two tissue
compartment model, a mathematical analysis about the necessity of measurements
of $C_\tot(t)$ and/or $f(t)$ for successful parameter identification, even in
an idealized, noiseless scenario, does not exist. In addition to that, even in
the presence of an arbitrary number of such measurements, it is not clear from
an analytical perspective if the tissue parameters in specific ODE-based
compartment models can be recovered uniquely.

The aim to this work is to answer these questions. Using a polyexponential
parametrization of the arterial plasma concentration, which is frequently used
in practice, we prove the following: Let $(K_1^i,k_2^i,k_3^i)$ be the kinetic
parameters of different anatomical regions $i=1,\ldots,n$ of the irreversible
two tissue compartment model, let $T$ be the number of time-points where PET
measurements of the tissue tracer concentration $C_\tis(t)$ are available, and
let $p$ be the degree of the polyexponential parametrization. Then, if
$T\geq2(p+3)$, and under some non-restrictive technical conditions as stated in
Theorem \ref{thm:main_uniqueness}, the parameters $k_2^i,k_3^i$ for
$i=1,\ldots,n$ can be identified uniquely already from the available
measurements of the tracer concentration in the different tissues without the
need for $C_\tot(t)$ and $f(t)$. Further, the $K_1^i$ can also be identified
already from these measurements up to a constant that is the same for all
regions $i$. In addition, the parameters $K_1^i$ can be identified exactly if a
sufficient number of measurements of the total arterial tracer concentration
$C_\tot$ is available, without the need for $f(t)$.

Besides these unique identifiability results that consider the idealized case
of a noise-free measurement, we also present analytical results for a standard
Tikhonov regularization approach that addresses the situation of noisy
measurements. Using classical results from regularization theory, we show that
the Tikhonov regularization approach is stable w.r.t. perturbations of the data
and, in the vanishing noise limit, allows to approximate the ground-truth
tissue parameters. Numerical experiments further illustrate our analytic
results also in an application example.

\noindent\textbf{Scope of the paper.} In Section \ref{sec:compartment_model}, we introduce the irreversible two tissue compartment model and provide basic results on explicit solutions both in the general case and in case the arterial concentration is parametrized via polyexponential functions. In Section \ref{sec:unique_identifiability} we present and prove our main results on unique identifiability of parameters. In Section \ref{sec:tikhonov_parameter_id} we introduce a Tikhonov regularization approach and show stability and consistency results, and in Sections \ref{sec:algorithm} and \ref{sec:experiments} we provide a numerical algorithm and numerical experiments for an application example.

\section{The irreversible two tissue compartment model} \label{sec:compartment_model}
The \emph{irreversible two tissue compartment model} describes the
interdependence of the concentration of a radio tracer in the arterial blood
plasma and in the extra-vascular compartment, where the latter is further
decomposed in a \emph{free} and a \emph{bound} compartment. Note that in the
irreversible model, once the radio tracer has reached the bound compartment, it
is trapped. A visualization of the model is provided in Figure
\ref{fig:comparment_model_scheme}.

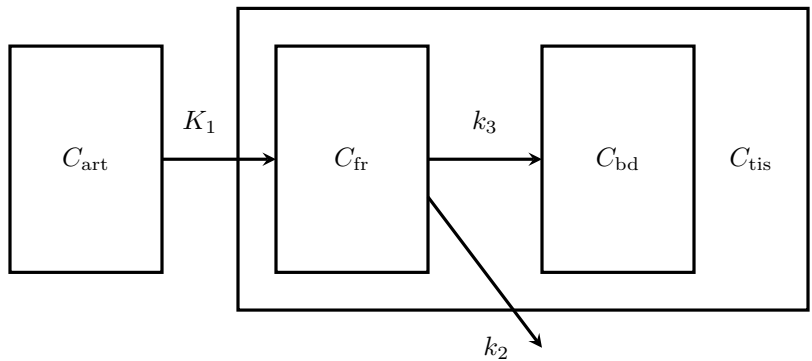
\begin{figure}
	\begin{center}
		\begin{tikzpicture}[%
				>=stealth,
				node distance=3cm,
				on grid,
				auto
			]
			\draw[very thick] (0,0.5) rectangle (2,3.5);
			\draw[very thick] (3,0) rectangle (10.5,4);
			\draw[very thick] (3.5,0.5) rectangle (5.5,3.5);
			\draw[very thick] (7,0.5) rectangle (9,3.5);

			\path[->,very thick] (2,2) edge (3.5,2);
			\path[->,very thick] (5.5,2) edge (7,2);
			\path[->,very thick] (5.5,1.5) edge (7,-0.5);

			\node (T) at (1,2)[align=center]{\(C_\art\)};;
			\node (T) at (4.5,2)[align=center]{\(C_\free\)};;
			\node (T) at (8,2)[align=center]{\(C_\bd\)};;
			\node (T) at (9.75,2)[align=center]{\(C_\tis\)};;

			\node (T) at (2.5,2.5)[align=center]{\(K_1\)};;
			\node (T) at (6.25,2.5)[align=center]{\(k_3\)};;
			\node (T) at (6.4,-0.5)[align=center]{\(k_2\)};;
		\end{tikzpicture}
	\end{center}
	\caption{\label{fig:comparment_model_scheme} Irreversible two tissue compartment model. The boxes around \(C_\art, C_\bd, C_\free\) and \(C_\tis\) represent the respective concentrations and the arrows between them the directional exchange with the respective rate depicted above or below the corresponding arrow. The box around the \(C_\bd\) and \(C_\free\) indicates the extra-vascular measurements associated with \(C_\tis\).}
\end{figure}

We denote by $C_\art:[0,\infty) \rightarrow [0,\infty)$ the arterial plasma
concentration of the non-metabolized PET tracer. Further, for any anatomical
region $i=1,\ldots,n$, we denote by $C^i_\free:[0,\infty) \rightarrow
	[0,\infty)$ and $C^i_\bd:[0,\infty) \rightarrow [0,\infty)$ the free and the
bound compartment of the tracer in region $i$, respectively, and by $C^i_\tis =
	C^i_\free + C^i_\bd$ we denote the sum of the two compartments in region $i$.
Using the \emph{irreversible two tissue compartment model}, the interaction of
these quantities is described by the following system of ordinary differential
equations (ODEs):
\begin{align*}
	\label{eq:odesystem}
	\begin{cases}
		\frac{\dx}{\dx t}C^i_\free= K^i_1 C_\art - \left(k^i_2+k^i_3\right) C^i_\free, & t>0 \\
		\frac{\dx}{\dx t}C^i_\bd = k^i_3 C^i_\free,                                    & t>0 \\
		C^i_\free\left(0\right)=0, ~ C^i_\bd\left(0\right)=0 \tag{S}
	\end{cases}
\end{align*}
Here, the parameters $K_1^i$, $k_2^i$ and $k_3^i$ are the tracer kinetic parameters that
define the interaction of the different compartments in region $i$.

Our goal is to identify the parameters $K_1^i$, $k_2^i$ and $k_3^i$ for each
$i=1,\ldots,n$. For this, we can use measurements of the $C_\tis^i(t_l)$ at
different time-points $t_1,\ldots,t_T$ obtained from reconstructed PET images
at different time points after tracer injection. Further, the parameter
identification typically relies on additional measurements related to $C_\art$.
Here, the standard procedure is to take arterial blood samples and to measure
the total activity concentration of the arterial blood samples, given as
$C_\tot:[0,\infty) \rightarrow [0,\infty)$. The relation of the total
concentration $C_\tot$ to the arterial plasma concentration $C_\art$ of
non-metabolized tracer is described via an unknown parent plasma fraction
function $f:[0,\infty) \rightarrow [0,1]$ with $f(0)=1$ as
\[ C_\art(t) = f(t)C_\tot(t) .\]
As described above, to obtain $f(t)$ and thus $C_\art(t)$, a time-consuming and
costly plasma separation and metabolite analysis of the blood samples has to be
performed.

In order to realize the parameter identification for the ODE model
\eqref{eq:odesystem} in practice, the involved functionals need to be
discretized, e.g., via a suitable parametrization. For the arterial
concentration $C_\art$, it is standard to use a parametrization via
polyexponential functions as defined in the following.
\begin{definition}[Polyexponential functions]
	We call a function \(g\) \index{Polyexponential!class} polyexponential of degree \(p\in\mathbb{N}\) if there exist \(\lambda_i, \mu_i\in\mathbb{R}\) for \(1\leq i\leq p\) where the \(\left(\mu_i\right)_{i=1}^p\) are pairwise distinct and \(\lambda_i\neq 0\) for all \(i\) such that
	\[
		g\left(t\right) = \sum_{i=1}^p\lambda_i e^{\mu_i t}.
	\]
	We write \(\deg\left(g\right)=p\) and call the zero-function polyexponential of
	degree zero. By \(\mathcal{P}_p\) we denote the set of polyexponential
	functions of degree less or equal to $p$, and by $\Pc = \cup_{p=0}^\infty
		\Pc_p$ the set of polyexponential functions (of any degree).
\end{definition}
\begin{remark}
	\label{polyprop}
	It obviously holds that $\Pc$ is a vector space, and even a subalgebra of \(\mathcal{C}^\infty\left(\mathbb{R}\right)\). It is also worth noting that, as direct consequence of the Stone-Weierstrass Theorem, polyexponential functions are dense in the set of continuous functions on compact domains. Thus, they are a reasonable approximation class also from the analytic perspective.
\end{remark}
Modeling $C_\art$ as polyexponential function already defines a parametrization of the resulting solutions of the ODE system \eqref{eq:odesystem}. As we will see in Lemma \ref{lem:solution_representation_polyexponential} below, the following notion of generalized polyexponential functions is the appropriate notion to describe such solutions.
\begin{definition}[Generalized Polyexponential class]
	\label{def:gen_poly_class}
	We call a function $g$ generalized polyexponential if it is of the form
	\[
		g\left(t\right)=P_1\left(t\right)e^{\mu_1 t}+\dots+P_l\left(t\right)e^{\mu_l t},
	\]
	where the \(P_1, \dots, P_l\not\equiv 0\) are polynomials of degree \(m_1-1,
	\dots, m_l-1\), respectively, and the $\mu_j$ are pairwise distinct constants.
	We denote the class of such generalized polyexponential functions with
	polynomials of degree at most $m_1-1,\ldots,m_l-1$ by
	\[
		\mathcal{P}\left[m_1, \dots, m_l\right].
	\]
	We define the degree of \(g\) by
	\[
		\deg\left(g\right)=m_1+\dots+m_l.
	\]
	In case \(m_1=\dots=m_l=1\) we write \(\mathcal{P}_{\deg\left(g\right)}\) for
	the resulting polyexponential class.
\end{definition}
The next two results, which follow from standard ODE theory, provide explicitly the solutions $(C_\free,C_\bd)$ of the ODE system \eqref{eq:odesystem}, once in the general case and once in case $C_\art$ is modeled as polyexponential function.
\begin{lemma} \label{lem:ode_solution_general}
	Let $C_\art:[0,\infty) \rightarrow [0,\infty)$ be continuous, and let the parameters $K_1^i$, $k_2^i$ and $k_3^i$ be fixed for $i=1,\ldots,n$ such that $k_2^i + k_3^i \neq 0$. Then, for each $i=1,\ldots,n$, the ODE system \eqref{eq:odesystem} admits a unique solution $(C_\free^i,C_\bd^i)$ that is defined on all of $[0,\infty)$, and such that $C_\tis^i = C_\free^i + C_\bd^i$ is given as
	\begin{equation}
		\label{eq:ct_representation_general}
		C_\tis^i = \frac{K_1^ik_2^i}{k_2^i+k_3^i}e^{-\left(k_2^i+k_3^i\right)t}\int_0^t e^{\left(k_2^i+k_3^i\right)s} C_\art\left(s\right)\dx s+\frac{K_1^ik_3^i}{k_2^i+k_3^i}\int_0^t C_\art\left(s\right)\dx s.
	\end{equation}
\end{lemma}
\begin{proof}
	Fix $i \in \{1,\ldots,n\}$. From the equation for $C_\free^i$ in \eqref{eq:odesystem} it immediately follows that
	\begin{equation}
		\label{solution_S}
		C_\free^i\left(t\right) = K_1^i e^{-\left(k_2^i+k_3^i\right)t}\int_0^t e^{\left(k_2^i+k_3^i\right)s}C_\art\left(s\right)\dx s.
	\end{equation}
	This in turn implies that
	\[
		C_\bd^i\left(t\right)= -\frac{K_1^i k_3^i}{k_2^i+k_3^i}e^{-\left(k_2^i+k_3^i\right)t}\int_0^t e^{\left(k_2^i+k_3^i\right)s}C_\art\left(s\right)\dx s+\frac{K_1^i k_3^i}{k_2^i+k_3^i}\int_0^t C_\art\left(s\right)\dx s
	\]
	and, consequently, that
	\[
		C_\tis^i(t) = \frac{K_1^ik_2^i}{k_2^i+k_3^i}e^{-\left(k_2^i+k_3^i\right)t}\int_0^t e^{\left(k_2^i+k_3^i\right)s} C_\art\left(s\right)\dx s+\frac{K_1^ik_3^i}{k_2^i+k_3^i}\int_0^t C_\art\left(s\right)\dx s
	\]
	as claimed.
\end{proof}

\begin{lemma}
	\label{lem:solution_representation_polyexponential}
	If $C_\art\in\mathcal{P}_p$ is given as
	\[ C_\art(t) = \sum_{j=1}^p \lambda_j e^{\mu_j t},
	\]
	then, for $i \in \{1,\ldots,n\}$, $C_\tis^i$ of Lemma
	\ref{lem:ode_solution_general} is given as
	\begin{multline}
		\label{ctextended}
		C_\tis^i(t) = \frac{K_1^i}{k_2^i+k_3^i}\sum_{j=1}^p\left( \frac{k_3^i}{\mu_j}\mathbf{1}_{\{\mu_j \neq 0\}}+\frac{k_2^i}{k_2^i+k_3^i+\mu_j}\mathbf{1}_{\left\{k_2^i+k_3^i+\mu_j\neq 0\right\}}\right)\lambda_je^{\mu_jt}
		\\-\left[\frac{K_1^ik_2^i}{k_2^i+k_3^i}\sum_{\substack{j=1 \\ k_2^i+k_3^i+\mu_j\neq 0}}^p\frac{\lambda_j}{k_2^i+k_3^i+\mu_j}\right]e^{-\left(k_2^i+k_3^i\right)t} -\frac{K_1^ik_3^i}{k_2^i+k_3^i}\sum_{\substack{j=1 \\ \mu_j\neq 0}}^p \frac{\lambda_j}{\mu_j} \\
		+\left[\frac{K_1^ik_2^i}{k_2^i+k_3^i}\sum_{\substack{j=1 \\ k_2^i+k_3^i+\mu_j= 0}}^p\lambda_j\right]te^{-\left(k_2^i+k_3^i\right)t} + \frac{K_1^ik_3^i}{k_2^i+k_3^i}\sum_{\substack{j=1 \\ \mu_j= 0}}^p  \lambda_j t
	\end{multline}
\end{lemma}
\begin{proof}
	This follows immediately by inserting the representation of $C_\art$ in \eqref{eq:ct_representation_general}.
\end{proof}
\begin{remark}[Sign of exponents $\mu_j$] \label{rem:sign_of_exponents}
	Note that for the ground truth arterial concentration $C_\art$, we will always have $\mu_j < 0$ (in particular $\mu_j \neq 0$) for all $j=1,\ldots,p$, since otherwise this would imply the unphysiological situation that $C_\art(t) \rightarrow c \neq 0$ for $t \rightarrow \infty$.
\end{remark}

\section{Unique identifiability} \label{sec:unique_identifiability}

In view of Lemma \ref{lem:solution_representation_polyexponential} from the
previous section, it is clear that the question of unique identifiability of
the parameters $K_1^i$, $k_2^i$ and $k_3^i$ from measurements of
$C_\tis^i(t_l)$ at time points $t_1,\ldots,t_T$ is related to the question of
uniqueness of interpolations with generalized polyexponential functions. A
first, existing result in that direction is as follows.

\begin{lemma}[Roots of generalized polyexponential functions]
	\label{lem:genpolyroots}
	Let \(P_1, \dots, P_l\) be polynomials of degree \(m_1-1, \dots, m_l-1\) such
	that at least one of them is not identically zero, and let the constants
	\(\mu_1, \dots, \mu_l\) be pairwise distinct. Then the function
	\[
		g\left(t\right)=P_1\left(t\right)e^{\mu_1 t}+\dots+P_l\left(t\right)e^{\mu_l t}
	\]
	admits at most \(m_1+\dots+m_l-1\) real roots.
\end{lemma}
\begin{proof}
	See \cite[Exercise 75 (p. 48)]{Polya2013}.
\end{proof}

As a consequence of the previous proposition, we now obtain the following
unique interpolation result for generalized polyexponential functions.
\begin{lemma}[Unique interpolation]
	\label{lem:polyexpinter}
	Let \(m_1,\dots, m_p, T\in \mathbb{N}\) be such that \[
		2\left(m_1+\dots+m_p\right)\leq T.\]
	Then, for any choice of tuples \(\left(t_i, s_i\right)\in \mathbb{R}^2\),
	$i=1,\ldots,T$, with \(t_1< \dots< t_T\), there exists at most one generalized
	polyexponential function \(h\in\mathcal{P}\left[m_1, \dots, m_p\right]\) such
	that
	\begin{equation}
		\label{interpolat1}
		h(t_l) = s_l
	\end{equation}
	for $l=1,\ldots,T$,  i.e., in case
	\[ h(t) = \sum_{j=1}^p P_j\left(t\right)e^{\mu_j t} \quad \text{and}\quad \tilde h(t) = \sum_{j=1}^p \tilde P_j\left(t\right)e^{\tilde \mu_j t}
	\]
	are two generalized polyexponential functions with
	$h,\tilde{h}\in\mathcal{P}\left[m_1, \dots, m_p\right]\) fulfilling the
		interpolation condition \eqref{interpolat1}, then, up to re-indexing,
		\[ P_j \equiv \tilde{P_j} \]
		for all $j$ and $\mu_j = \tilde{\mu}_j$ for all $j$ where $P_j \not\equiv 0$.
\end{lemma}
\begin{proof}
	Let both $h, \tilde{h}\in\mathcal{P}\left[m_1, \dots, m_p\right]\) fulfill the interpolation condition \eqref{interpolat1}, and, w.l.o.g., assume that $P_j \not \equiv 0$ and $\tilde{P}_j \not \equiv 0$ for all $j$. Then, $h - \tilde{h} \in \mathcal{P}\left[m_1, \dots, m_p,m_1,\ldots,m_p\right]$ and $(h-\tilde{h})(t_l) = 0$ for $l=1,\ldots,T$. Lemma \ref{lem:genpolyroots} then implies that all polynomials appearing in $h-\tilde{h}$ in a representation as in Definition \ref{def:gen_poly_class} are identically zero.
		This implies that the $(P_j)_{j=1}^p$ and $(\tilde P_j)_{j=1}^p$ coincide up to re-indexing and, likewise, that the corresponding coefficient $(\mu_j)_{j=1}^p$ and $(\tilde \mu_j)_{j=1}^p$ where the corresponding polynomials are non-zero coincide as well.
\end{proof}

Based on this result, we now address the question of uniquely identifying the
parameters of the ODE system \eqref{eq:odesystem} from time-discrete
measurements $C_\tis^i(t_1), \ldots,C_\tis^i(t_T)$ with $i=1,\ldots,n$ and
measurements $C_\tot(s_1),\ldots,C_\tot(s_q)$. For this, we first introduce the
following notation.
\begin{definition}[Parameter configuration] We call the parameters $p,n\in \N$, $((\lambda_j,\mu_j))_{j=1}^p \in \R^{2 \times p} $, $((K_1^i,k_2^i,k_3^i))_{i=1}^n \in \R^{3 \times n}$ together with the functions $(C^i_\tis)_{i=1}^n$ and
	\[ C_\art(t) = \sum_{j=1}^p \lambda_j e^{\mu_jt} \]
	a configuration of the irreversible two tissue compartment model if $\lambda_j
		\neq 0$ for $j=1,\ldots,p$, the $\mu_j$, $j=1,\ldots,p$ are pairwise distinct,
	and, for $i=1,\ldots,n$, $C_\tis^i = C_\free^i + C_\bd^i$ with
	$(C_\free^i,C_\bd^i)$ the solution of the ODE system \eqref{eq:odesystem} with
	arterial concentration $C_\art$ and parameters $K_1^i,k_2^i,k_3^i$.

\end{definition}

Central for our unique identifiability result will be the following technical
assumption on a parameter configuration $(p,n,((\lambda_j,\mu_j))_{j=1}^p
	,((K_1^i,k_2^i,k_3^i))_{i=1}^n,(C^i_\tis)_{i=1}^n,C_\art)$.
\begin{equation} \label{ass:main_technical_ass} \tag{A}
	\left\{
	\begin{aligned}
		 & \text{For any $j_0$, there are at least three regions } i_s, \, s=1,\ldots,3, \text{ where }                  \\
		 & k^{i_s}_3 \text{ and } k^{i_s}_2 + k^{i_s}_3 \text{ are each pairwise distinct, }\mu_{j_0} + k_3^{i_s} \neq 0 \\
		 & \text{and either }\mu_{j_0} + k_2^{i_s} + k_3^{i_s} = 0 \text{ or } \sum_{\substack{j=1                       \\ \mu_{j} + k_2^{i_s} + k_3^{i_s} \neq 0}}^p \frac{\lambda_j}{k_2^{i_s} + k_3^{i_s} + \mu_j} \neq 0
	\end{aligned}
	\right.
\end{equation}
As the following lemma shows, this assumption holds in case our measurement setup comprises sufficiently many regions where the parameters $k_3^i$ and $k_2^i+k_3^i$ are pairwise distinct. This is reasonable to assume in practice, and also it is a condition which is to be expected: Our unique identifiability result will require a sufficient amount of different regions, and different regions with the same tissue parameter do not provide any additional information on the dynamics of the ODE model.
\begin{lemma} \label{lem:simpler_version_of_main_assumption}
	Assume that there are at least $p+3$ regions $i_{1},\ldots,i_{p+3}$, with $p\geq 1$, where each the $k^{i_s}_3$ and the $k^{i_s}_2 + k^{i_s}_3$ are pairwise distinct for $s=1,\ldots,p+3$. Then Assumption \eqref{ass:main_technical_ass} holds.
	\begin{proof}
		For $z \in \R$, note that
		\[
			\left( \prod_{\substack{j=1 \\ \mu_{j} + z \neq 0}}^p {z + \mu_j}  \right)\left( \sum_{\substack{i=1 \\ \mu_{i} + z\neq 0}}^p \frac{\lambda_i}{z  + \mu_i} \right) =
			\sum_{\substack{i=1 \\ \mu_{i} + z \neq 0}}^p \lambda _i
			\prod_{\substack{j=1 \\ j \neq i \\ \mu_{j} + z \neq 0}}^p (z + \mu_j)
		\]
		is a polynomial in $z$ of degree at most $p-1$. Hence it can admit at most
		$p-1$ distinct roots. Now since there are at least $p+3$ regions where each the
		$k^{i_s}_3$ and the $k^{i_s}_2 + k^{i_s}_3$ are pairwise distinct, for at least
		four of them, say $i_1,\ldots,i_4$, $z=k^{i_s}_2 + k^{i_s}_3$ cannot be a root
		of the above polynomial. Further, for those four regions, since the $k^{i_s}_3$
		are pairwise distinct, for any given $\mu_{j_0}$, at most one can be such that
		$\mu_{j_0} + k_3^{i_s} = 0$. As a consequence, the remaining three are such
		that the conditions of Assumption \eqref{ass:main_technical_ass} hold.
	\end{proof}
\end{lemma}
Based on Assumption \eqref{ass:main_technical_ass}, we now obtain the following proposition, which is the technical basis for our subsequent results on unique identifiability.
\begin{proposition} \label{prop:main_uniqueness_result}
	Let $(p,n,((\lambda_j,\mu_j))_{j=1}^p ,((K_1^i,k_2^i,k_3^i))_{i=1}^n,(C^i_\tis)_{i=1}^n,C_\art)$ be a configuration of the irreversible two tissue compartment model with $\mu_j \neq 0$ for $j=1,\ldots,p$, $p \geq 3$, $n \geq 3$, $K_1^i,k_2^i,k_3^i > 0$ for all $i=1,\ldots,n$ and such that Assumption \eqref{ass:main_technical_ass} holds. Let further $t_1,\ldots,t_T$ be distinct points such that
	\[ T \geq 2(p+3) .\]
	Then, with $(\tilde p,n,((\tilde \lambda_j,\tilde \mu_j))_{j=1}^{\tilde p}
		,((\tilde K_1^i,\tilde k_2^i,\tilde k_3^i))_{i=1}^n,(\tilde
		C^i_\tis)_{i=1}^n,\tilde C_\art)$ any other configuration of the irreversible
	two tissue compartment model such that $\tilde{p}\leq p$, $\tilde k_3^i \neq 0$
	and $\tilde k_2^i+\tilde k_3^i \neq 0$ for all $i=1,\ldots,n$, it follows from
	$C_\tis(t_l) = \tilde C_\tis(t_l)$ for $l=1,\ldots, T$ that
	\[ \tilde k_2^i = k_2^i \text{ and } \tilde k_3^i = k_3^i
	\]
	for all $i=1,\ldots,n$, that there exists a constant $\zeta \neq 0$ such that
	\[ K_1^i = \zeta  \tilde  K_1^i \]
	for all $i=1,\ldots,n$, that $p= \tilde{p}$ and that (up to re-indexing)
	\[ \tilde \mu_j = \mu_j \text{ and } \tilde \lambda_j = \zeta \lambda_j \text{ for all }j=1,\ldots,p.
	\]
\end{proposition}
\begin{proof}
	Take $(p,n,((\lambda_j,\mu_j))_{j=1}^p ,((K_1^i,k_2^i,k_3^i))_{i=1}^n,(C^i_\tis)_{i=1}^n,C_\art)$ and $(\tilde p,n,((\tilde \lambda_j,\tilde \mu_j))_{j=1}^{\tilde p},$

	\noindent $((\tilde K_1^i,\tilde k_2^i,\tilde k_3^i))_{i=1}^n,(\tilde C^i_\tis)_{i=1}^n,\tilde C_\art)$ to be two configurations as stated in the proposition, such that in particular
	\begin{equation}
		\label{eq:uniquness_prop_measurment_equality}
		C_\tis^i(t_l) =  \tilde C_\tis^i(t_l)
	\end{equation}
	for $l=1,\ldots,T$.

	Now since $\mu_j \neq 0$ for all $j=1,\ldots,p$, we obtain the following
	representation of $C_\tis^i$:
	\begin{multline}
		\label{ctextended_mu_nonzero}
		C_\tis^i(t) = \frac{K_1^i}{k_2^i+k_3^i}\sum_{j=1}^p\left( \frac{k_3^i}{\mu_j}+\frac{k_2^i}{k_2^i+k_3^i+\mu_j}\mathbf{1}_{\left\{k_2^i+k_3^i+\mu_j\neq 0\right\}}\right)\lambda_je^{\mu_jt}
		\\-\left[\frac{K_1^ik_2^i}{k_2^i+k_3^i}\sum_{\substack{j=1 \\ k_2^i+k_3^i+\mu_j\neq 0}}^p\frac{\lambda_j}{k_2^i+k_3^i+\mu_j}\right]e^{-\left(k_2^i+k_3^i\right)t} -\frac{K_1^ik_3^i}{k_2^i+k_3^i}\sum_{j=1}^p \frac{\lambda_j}{\mu_j} \\
		+\left[\frac{K_1^ik_2^i}{k_2^i+k_3^i}\sum_{\substack{j=1 \\ k_2^i+k_3^i+\mu_j= 0}}^p\lambda_j\right]te^{-\left(k_2^i+k_3^i\right)t} .
	\end{multline}
	In particular, for any region $i \in \{1,\ldots,n\}$, the coefficients of $e^{\mu_j t}$ for $j=1,\ldots,p$ in this representation are given as either
	\[
		\frac{K_1^{i}\lambda_{j}k_3^{i}}{\mu_{j}\left(k_2^{i}+k_3^{i}\right)} \neq 0
	\]
	in case $k_2^{i}+k_3^{i}+\mu_{j} = 0$ or
	\[
		\frac{K_1^{i}\lambda_{j}\left(\mu_{j}+k_3^{i}\right)}{\mu_{j}\left(k_2^{i}+k_3^{i}+\mu_{j}\right)}
	\]
	otherwise. The latter can only be zero if $\mu_j + k_3^{i_0}=0$, which can
	happen for at most one $\hat{j}$ by the $(\mu_j)_j$ being pairwise distinct.
	Since $p\geq 3$ by assumption, this implies in particular that $C_\tis^i$ is a
	non-zero function for any $i$. As a consequence of
	\eqref{eq:uniquness_prop_measurment_equality}, the condition \(T\geq 2(p+3)\) and the unique interpolation
	result of Lemma \ref{lem:polyexpinter}, this implies that $\tilde C_\tis^i$ is
	a non-zero function, such that in particular $\tilde K_1^i\neq 0$ for all $i\in
		\{1,\ldots,n\}$. Together with the assumption that $\tilde k _3^i \neq 0$ for
	all $i$, we also obtain that $\tilde \mu_j \neq 0$ for all $j$, since otherwise
	$\tilde C_\tis^i$ would have a non-zero coefficient of $t$. As a consequence,
	also $\tilde C_\tis^i$ admits a representation as in
	\eqref{ctextended_mu_nonzero}.

	\textbf{Uniqueness of the exponents $(\mu_j)_{j=1}^p$.} As first step, we now aim to show that $\tilde{p} = p$ (in particular $\tilde{\lambda}_j\neq 0$ for all $j$) and that (up to re-indexing) $\mu_j=\tilde{\mu}_j$ for all $j=1,\ldots,p$.

	We start with a region $i_0 \in \{1,\ldots,n\}$. In this region, as argued
	above, the coefficients of the $e^{\mu_j t}$ for $j=1,\ldots,p$ can be zero for
	at most one $\hat{j}$. Since further at most one $j_0$ can be such that
	$\mu_{j_0} = -(\tilde k_2^{i_0} + \tilde k_3^{i_0})$, the unique interpolation
	result of Lemma \ref{lem:polyexpinter} applied to $C_\tis^{i_0}$ and $\tilde
		C_\tis^{i_0}$ yields that $\tilde{p}\geq p-2 \geq 1$ and that (up to
	re-indexing) $\mu_j=\tilde{\mu}_j$ for all $j \notin \{\hat{j},j_0\}$.

	Now as a consequence of Assumption \ref{ass:main_technical_ass}, we can pick a
	region $i_1\neq i_0$ with $k_3^{i_1} \neq k_3^{i_0}$ where $\mu_{j_0} +
		k_3^{i_1}\neq 0$. Since already $\mu_{\hat j} + k_3^{i_0}=0$ it further must
	hold that $\mu_{\hat j} + k_3^{i_1}\neq 0$. This means that the coefficients of
	both $e^{\mu_{\hat{j}}t}$ and $e^{\mu_{j_0}t}$ in the representation of
	$C_\tis^{i_1}$ as in \eqref{ctextended} are non-zero. Again by the $\mu_j$
	being pairwise distinct, this implies that $\tilde{p}\geq p-1$ and that (up to
	re-indexing) either $\mu_{\hat{j}} = \tilde \mu_{\hat{j}}$ or $\mu_{j_0} =
		\tilde \mu_{j_0}$.

	\noindent \textbf{Case I.}  Assume that $\mu_{\hat{j}} = \tilde \mu_{\hat{j}}$. If also $\mu_{j_0}  = \tilde{\mu}_{j_0}$ (and $\tilde{p} \geq p$) we are done with this step, so assume the contrary.

	Now as a consequence of Assumption \ref{ass:main_technical_ass}, we can pick
	$i_2,i_3$ and $i_4$ to be regions where $\mu_{j_0} + k^{i_l}_3 \neq 0$ and
	either $\mu_{j_0} + k_2^{i_l} + k_3^{i_l} = 0$ or the coefficient of
	$e^{-(k_2^{i_l} + k_3^{i_l})t}$ in the representation of $C_\tis^{i_l}$ is
	non-zero.
	The fact that $\mu_{j_0} + k^{i_l}_3 \neq 0$, together with $\mu_{j_0} \neq
		\tilde{\mu}_{j_0}$ by assumption, further yields that $\mu_{j_0} = -(\tilde
		k_2^{i_l} + \tilde k_3^{i_l})$ for all $l=2,3,4$.

	Now we argue that in each region $i_l$ with $l=2,3,4$, it must hold that either
	$-(k_2^{i_l} + k_3^{i_l}) = \tilde\mu_{j_0} $ or $k_2^{i_l} + k_3^{i_l} = \tilde
		k_2^{i_l} + \tilde k_3^{i_l}$. To this aim, we make another case distinction
	for a fixed $l \in \{2,3,4\}$.

	\noindent \textbf{Case I.A}  Assume that there exists $j_l$ with $k_2^{i_l} + k_3^{i_l} + \mu_{j_l} = 0$. From the fact that this can happen at most for one $j_l$ and that $\lambda_{j_l} \neq 0$, it follows that the coefficient of $te^{-(k_2^{i_l} + k_3^{i_l})t}$ is non-zero. Consequently, it follows from the unique interpolation result that $k_2^{i_l} + k_3^{i_l} = \tilde k_2^{i_l} + \tilde k_3^{i_l}$ as claimed.

	\noindent \textbf{Case I.B}  Assume that $k_2^{i_l} + k_3^{i_l} + \mu_{j} \neq 0$ for all $j$. This means that $\tilde{\mu}_j = \mu_j \neq -(k_2^{i_l} + k_3^{i_l})$ for all $j\neq j_0$. But since the coefficient of $e^{-(k_2^{i_l} + k_3^{i_l})t}$ is non-zero, by the unique interpolation result it must hence hold that either $-(k_2^{i_l} + k_3^{i_l}) =  \tilde\mu_{j_0} $ or $k_2^{i_l} + k_3^{i_l} = \tilde k_2^{i_l} + \tilde k_3^{i_l}$ as claimed. This concludes Case I.B.

	Now given that either $-(k_2^{i_l} + k_3^{i_l}) = \tilde\mu_{j_0} $ or
	$k_2^{i_l} + k_3^{i_l} = \tilde k_2^{i_l} + \tilde k_3^{i_l}$ for all
	$l=2,3,4$, one of the two cases must happen at least twice. By uniqueness of
	the $k_2^{i_l} + k_3^{i_l} $ for $l=2,3,4$, only $k_2^{i_l} + k_3^{i_l} =
		\tilde k_2^{i_l} + \tilde k_3^{i_l}$ can happen twice. On the other hand, since
	$\mu_{j_0} = -(\tilde k_2^{i_l} + \tilde k_3^{i_l})$ for all $l=2,3,4$, this
	yields that at least two $k_2^{i_l}+k_3^{i_l}$ coincide, which is a
	contradiction. Hence Case I is complete.

	\noindent \textbf{Case II.}  Assume that $\mu_{j_0} = \tilde \mu_{j_0}$. In this case, interchanging the role of $\mu_{j_0}$ and $\mu_{\hat{j}}$, we can argue that $\mu_{\hat{j}} = \tilde \mu_{\hat{j}}$ exactly as in Case I.

	\textbf{Uniqueness of at least three of the exponents $k_2^{i}+k_3^{i}$.}
	Let $i_0$  be any region such that either $\mu_{j_0} + k_2^{i_0}+ k_3^{i_0} = 0$ for some $j_0 \in \{1,\ldots,p\}$ (such that the coefficient of $t e^{-(k_2^{i_l} + k_3^{i_l})t}$ in the representation of $C_\tis^{i_0}$ is non-zero) or the coefficient of $e^{-(k_2^{i_0} + k_3^{i_0})t}$ in the representation of $C_\tis^{i_0}$ is non-zero, and note that, according to Assumption \ref{ass:main_technical_ass}, at least three such regions exist.

	\noindent \textbf{Case I.} Assume that there exists $j_0 \in \{1,\ldots,p\}$ such that $k_2^{i_0} + k_3^{i_0} + \mu_{j_0} = 0$. This implies that the coefficient of $te^{-(k_2^{i_0} + k_3^{i_0})t}$ is non-zero and, consequently, already that $k_2^{i_0} + k_3^{i_0} = \tilde k_2^{i_0} + \tilde k_3^{i_0}$ by uniqueness of exponents.

	\noindent \textbf{Case II.} Assume that $k_2^{i_0} + k_3^{i_0} + \mu_{j} \neq 0$ for all $j=1,\ldots,p$. Now since then the coefficient of $e^{-(k_2^{i_0} + k_3^{i_0})t}$ is non-zero by assumption, $-(k_2^{i_0} + k_3^{i_0})$ must match some exponent in the representation of $\tilde C^{i_0}_\tis$. It cannot match any of the $\tilde{\mu}_j = \mu_j $ since $k_2^{i_0} + k_3^{i_0} + \mu_{j} \neq 0$ for all $j=1,\ldots,p$, hence again $k_2^{i_0} + k_3^{i_0} = \tilde k_2^{i_0} + \tilde k_3^{i_0}$ follows.

	\textbf{Uniqueness of at least three of the exponents $k_2^i,k_3^i$}.
	First note that for any $i \in \{1,\ldots,n\}$ where $\tilde k_2^{i} + \tilde k_3^{i} = k_2^{i} + k_3^{i} $, from the unique interpolation result, it follows that
	\begin{equation} \label{eq:coefficient_basic_equality}
		K^{i}_1 \lambda_j(\mu_j + k^ {i}_3) = \tilde K^{i}_1 \tilde \lambda_j(\mu_j + \tilde k^ {i}_3)
	\end{equation}
	for all $j=1,\ldots,p$. Indeed, in case $k_2^{i} + k_3^{i} + \mu_j = 0$, it follows from the coefficients of $t e^{-(k_2^{i} + k_3^{i})t}$ in $C_\tis^i$ and $\tilde C_\tis^i$ being equal that
	\[ \frac{K_1^{i} k_2^{i}\lambda_j}{k_2^{i}+k_3^{i}}  = \frac{\tilde K_1^{i} \tilde k_2^{i}\tilde \lambda_j}{k_2^{i}+k_3^{i}} ,
	\]
	which implies that $K_2^{i}k_2^{i} \lambda _j = \tilde K_2^{i} \tilde k_2^{i}
		\tilde \lambda _j$ and, using that $k_2^{i} = -\mu_j - k_3^{i}$ and
	$\tilde{k}_2^{i} = -\mu_j - \tilde{k}_3^{i}$, further yields $K^{{i}}_1
		\lambda_j(\mu_j + k^ {{i}}_3) = \tilde K^{{i}}_1 \tilde \lambda_j(\mu_j +
		\tilde k^ {{i}}_3) $ as claimed.

	In the other case, the equality \eqref{eq:coefficient_basic_equality} follows
	directly from the coefficients of $e^{\mu_j t}$ in $C_\tis^i$ and $\tilde
		C_\tis^i$ being equal.

	Now let $i_0$ be any region where $\tilde k_2^{i_0} + \tilde k_3^{i_0} =
		k_2^{i_0} + k_3^{i_0} $, and for which we want to show that $k_2^{i_0} = \tilde
		k_2^{i_0} $ and $k_3^{i_0} = \tilde k_3^{i_0} $. Again we consider several
	cases.

	\noindent \textbf{Case I.} Assume that there exists $j_0 \in \{1,\ldots,p\}$ such that $ \mu_{j_0} + k_3^{i_0}= 0$. In this case, it follows from \eqref{eq:coefficient_basic_equality} that also $\mu_{j_0} + \tilde k_3^{i_0} = 0$ (note that $\tilde \lambda_{j_0} \neq 0 $ and $\tilde K_1^{i_0} \neq 0$ since $\tilde{p}=p$), hence $k_3^{i_0} = \tilde k_3^{i_0}$ and, consequently, $k_2^{i_0} = \tilde k_2^{i_0}$ holds.

	\noindent \textbf{Case II.} Assume that $ \mu_{j} + k_3^{i_0} \neq 0$ for all $j$. In this case, using Assumption \ref{ass:main_technical_ass} and the previous step, we can select $i_1$ to be a second region where again $\tilde k_2^{i_1} + \tilde k_3^{i_1} = k_2^{i_1} + k_3^{i_1} $ and such that the $k_3^{i_0} \neq k_3^{i_1}$. We have two cases.

	\noindent \textbf{Case II.A} Assume that there exists $j_1 \in \{1,\ldots,p\}$ such that $ \mu_{j_1} + k_3^{i_1}= 0$. As in Case I above, this implies that  $k_3^{i_1} = \tilde k_3^{i_1}$. Further, choosing two indices $j_2,j_3 \in \{1,\ldots,p\}$ such that $j_1,j_2,j_3$ are pairwise distinct, it follows that $ \mu_{j_2} + k_3^{i_1} \neq 0$ and $ \mu_{j_3} + k_3^{i_1} \neq  0$ by the $\mu_j$ being different. Using  \eqref{eq:coefficient_basic_equality} and $k_3^{i_1} = \tilde k_3^{i_1}$ this implies
	\[
		K_1^{i_1}\lambda_{j_1} = \tilde K_1^{i_1} \tilde \lambda_{j_1} \qquad K_1^{i_1}\lambda_{j_2} = \tilde K_1^{i_1} \tilde \lambda_{j_2}
	\]
	Using that the $ \tilde K^{i_1}_1, \tilde K^{i_2}_1$ cannot be zero, these two
	equations imply
	\[ \frac{\tilde \lambda_{j_1}}{\lambda_{j_1}} = \frac{\tilde \lambda_{j_2}}{\lambda_{j_2}}.
	\]
	Combining this with the equations \eqref{eq:coefficient_basic_equality} for
	$i=i_0$ and $j=j_2,j_3$ we obtain
	\[
		\frac{\mu_{j_3}+\tilde k_3^{i_0}}{\mu_{j_3}+k_3^{i_0}} = \frac{\mu_{j_2}+\tilde k_3^{i_0}}{\mu_{j_2}+k_3^{i_0}}
	\]
	Reformulating this equation and using that $\mu_{j_2} \neq \mu_{j_3}$ this
	implies that $k_3^{i_0} = \tilde k_3^{i_0}$ and, consequently, $k_2^{i_0} =
		\tilde k_2^{i_0}$ holds.

	\noindent \textbf{Case II.B} Assume that $ \mu_{j} + k_3^{i_1} \neq 0$ for all $j$. Defining $\Lambda_j = \tilde{\lambda}_j /\lambda_j$, we then obtain from \eqref{eq:coefficient_basic_equality} for pairwise distinct $j_1,j_2,j_3\in\{1,\dots, p\}$ that
	\begin{align}
		\Lambda_{j_1} \frac{\mu_{j_1}+\tilde{k_3^{i_s}}}{\mu_{j_1}+k_3^{i_s}} = \Lambda_{j_2} \frac{\mu_{j_2}+\tilde{k_3^{i_s}}}{\mu_{j_2}+k_3^{i_s}} = \Lambda_{j_3} \frac{\mu_{j_3}+\tilde{k_3^{i_s}}}{\mu_{j_3}+k_3^{i_s}}.\notag
	\end{align}
	for $s=0,1$. From this, we conclude that

	\begin{equation} \label{eq:coeff_eq_inter}
		0= \frac{\mu_{j_r}+\tilde k_3^{i_0}}{\mu_{j_r}+ k_3^{i_0}}\frac{\mu_{j_s}+\tilde k_3^{i_1}}{\mu_{j_s}+ k_3^{i_1}}
		-  \frac{\mu_{j_r}+\tilde k_3^{i_1}}{\mu_{j_r}+ k_3^{i_1}}\frac{\mu_{j_s}+\tilde k_3^{i_0}}{\mu_{j_s}+ k_3^{i_0}}
	\end{equation}
	for $r,s \in \{1,2,3\}$ with $r\neq s$.

	Multiplying \eqref{eq:coeff_eq_inter} with the denominator $(\mu_{j_r}+
		k_3^{i_0})(\mu_{j_s}+ k_3^{i_1})(\mu_{j_r}+ k_3^{i_1})(\mu_{j_s}+ k_3^{i_0})$
	and further dividing by $(\mu_{j_r} - \mu_{j_s})$ we obtain

	\begin{multline*}
		0 = \mu_{j_r}\mu_{j_s}\left(\tilde k_3^{i_0}-\tilde k_3^{i_1} +k_3^{i_1}-k_3^{i_0}\right)+\left(\mu_{j_r}+\mu_{j_s}\right)\left(k_3^{i_1}\tilde k_3^{i_0 }-k_3^{i_0}\tilde k_3^{i_1}\right)\\+\left(k_3^{i_1}-k_3^{i_0}\right)\tilde k_3^{i_0}\tilde k_3^{i_1}+\left(\tilde k_3^{i_0}-\tilde k_3^{i_1}\right)k_3^{i_0} k_3^{i_1}
	\end{multline*}
	for $r,s \in \{1,2,3\}$ with $r\neq s$. Subtracting the above equation for $(r,s) = (1,3)$ from the same equation for $(r,s) = (1,2)$ and dividing by $(\mu_{j_2} - \mu_{j_3})$ we obtain
	\begin{equation} \label{eq:coeff_intermediate_equation}
		0 = \mu_{j_1}\left(\tilde k_3^{i_0}-\tilde k_3^{i_1}+k_3^{i_1}-k_3^{i_0}\right)+\left(k_3^{i_1}\tilde k_3^{i_0}-k_3^{i_0}\tilde k_3^{i_1}\right).
	\end{equation}
	Similarly, subtracting the above equation for $(r,s) = (2,3)$ from the same equation for $(r,s) = (2,1)$  and dividing by $(\mu_{j_1} - \mu_{j_3})$ we obtain
	\begin{equation}
		0 = \mu_{j_2}\left(\tilde k_3^{i_0}-\tilde k_3^{i_1}+k_3^{i_1}-k_3^{i_0}\right)+\left(k_3^{i_1}\tilde k_3^{i_0 }-k_3^{i_0}\tilde k_3^{i_1}\right).
	\end{equation}
	Combining the last two equations and using that $\mu_{j_1} \neq \mu_{j_2}$ we obtain
	\[
		\tilde k_3^{i_0}-k_3^{i_0} = \tilde k_3^{i_1} - k_3^{i_1},
	\]
	i.e., $\tilde k_3^{i_0} = k_3^{i_0} + \epsilon $ and $\tilde k_3^{i_1} =
		k_3^{i_1} + \epsilon $ for $\epsilon\in \R$. Inserting this into
	\eqref{eq:coeff_intermediate_equation} we obtain
	\[ \epsilon(k_3^{i_1} - k_3^{i_0}) = 0  \] which, together with $k_3^{i_1} \neq k_3^{i_0}$, yields $\epsilon = 0$ and
	hence in particular $k_3^{i_0} = \tilde{k}_3^{i_0}$ as desired. Together with
	$k_2^{i_0} + k_3^{i_0} = \tilde{k}_2^{i_0} + \tilde{k}_3^{i_0}$ this yields
	that also $k_2^{i_0} = \tilde{k}_2^{i_0}$.

	\textbf{Uniqueness of the remaining $k_2^i,k_3^i$ and of the $K_1^i$ up to a constant factor.}
	Take $i_0 $ to be a region where $\tilde k_2^{i_0} = k_2^{i_0}$ (we know already that such a region exists). It then follows from \eqref{eq:coefficient_basic_equality} that
	\begin{equation} \label{eq:coefficient_equality_interestimate}
		K^{i_0}_1 \lambda_j = \tilde K^{i_0}_1 \tilde \lambda_j.
	\end{equation}
	for $j=1,\ldots,p$. Thus, with $\zeta := K^{i_0}_1 / \tilde  K^{i_0}_1\neq 0$, we have that $\tilde \lambda_j = \zeta  \lambda_j$ for all $j$. We now aim to show that, for all $i \in \{1,\ldots,n\}$, $k_2^i = \tilde k_2^i$, $k_3^i = \tilde k_3^i$ and $K_1^i = \zeta \tilde K_1^i$.

	Consider $i \in \{1,\ldots,n\}$ fixed. To simplify notation, we drop here the
	index $i$, e.g., we write $K_1 = K_1 ^i$, $k_2 = k_2^i$ and $k_3 = k_3^i$ and
	similar for $\tilde K_1, \tilde k_2, \tilde k_3$.

	In case $k_2 + k_3 + \mu_{j_0} = 0$ for some $j_0$, we know already from the
	previous step that $k_2 = \tilde k_2$ and $k_3 = \tilde k_3$, such that, from
	equating coefficients in the representations of $C_\tis$ and $\tilde C_\tis$,
	we get
	\[
		K_1 \lambda_j =  \tilde K_1 \tilde  \lambda_j= \tilde K_1 \zeta  \lambda_j,
	\]
	such that also $K_1 = \zeta \tilde K_1$ as desired.

	In the other case that $k_2 + k_3 + \mu_j \neq 0$ for all $j$, we get from
	equating coefficients in the representations of $C_\tis$ and $\tilde C_\tis$,
	using $\tilde \lambda_j = \zeta \lambda_j$, that
	\begin{equation} \label{eq:remaining_coefficients_basic_equation}
		\frac{\tilde K_1\zeta  (\mu_j +\tilde  k_3)}{\tilde k_2 + \tilde k_3 + \mu_j} =
		\frac{K_1(\mu_j + k_3)}{ k_2 +  k_3 + \mu_j} := z_j
	\end{equation}
	for $j=1,\ldots,p$, where the $z_j$ are pairwise distinct by the $\mu_j$ being pairwise distinct. Now we show that, from \eqref{eq:remaining_coefficients_basic_equation}, it follows that $\zeta \tilde K_1 = K_1$, $\tilde k_2 = k_2$ and $ \tilde k_3 = k_3$. For this, we again need to distinguish several cases.

	\noindent \textbf{Case I.} $\tilde{k}_3+\mu_{j_0}=0$  for at least one $j_0 \in \{1,2,3\}$. This implies that also $k_3 + \mu_{j_0} =0$ and hence that $\tilde k_3 = k_3$. Considering $j_1,j_2 \in \{1,2,3\} \setminus \{j_0\}$ with $j_1 \neq j_2 $ it follows from the $\mu_j$ being pairwise distinct that $k_3 + \mu_{j_s} \neq 0$ for $s=1,2$, which implies that also $\tilde k_3  + \mu_{j_s} \neq 0$ for $s=1,2$ and, consequently, that
	\begin{equation}
	\label{zeta_tilde_K1_representation}
	\zeta  \tilde K_1 = \frac{z_{j_s}}{\mu_{j_s} + k_3}\tilde k_2 + \frac{z_{j_s}}{\mu_{j_s} + k_3}(\mu_{j_s}+k_3)
	\end{equation}
	for $s=1,2$. Now if $\frac{z_{j_1}}{\mu_{j_1} + k_3} \neq
		\frac{z_{j_2}}{\mu_{j_2} + k_3}$, one may derive $\tilde k_2 = k_2$ by rearranging the terms in \eqref{zeta_tilde_K1_representation} for $s=1, 2$. Hence, by inserting the obtained equalities $\tilde k_2 = k_2$ and $\tilde k_3 =k_3$ in \eqref{eq:remaining_coefficients_basic_equation} we further deduce $\zeta \tilde K_1 = K_1$. If, on the other hand $\frac{z_{j_1}}{\mu_{j_1} +
			k_3} = \frac{z_{j_2}}{\mu_{j_2} + k_3}$ we can plug in the definition of
	$z_{j_1}, z_{j_2}$ and obtain
	\[
		\frac{K_1}{ k_2 +  k_3 + \mu_{j_1}} = \frac{K_1}{ k_2 +  k_3 + \mu_{j_2}},
	\]
	which yields $\mu_{j_1} = \mu_{j_2}$ and hence a contradiction.

	\noindent \textbf{Case II.} $\tilde k_3 + \mu_j \neq 0$ for all $j=1,2,3$. 
In this case we can reformulate \eqref{eq:remaining_coefficients_basic_equation} to obtain
	\begin{equation} \label{eq:remaining_coefficients_reformulated_equation}
		\zeta \tilde K_1  = z_j \frac{\tilde k_2 + \tilde k_3 + \mu_j}{\mu_j +\tilde  k_3}
	\end{equation}
	for all $j=1,2,3$. In particular, this yields
	\[
		z_1 \frac{\tilde k_2 + \tilde k_3 + \mu_1}{\mu_1 +\tilde  k_3} = z_2 \frac{\tilde k_2 + \tilde k_3 + \mu_2}{\mu_2 +\tilde  k_3}.
	\]
	Now if $z_1(\mu_{2} + \tilde{k}_3) =  z_2(\mu_{1} + \tilde{k}_3)$, this implies $\mu_1=\mu_2$ and hence a contradiction. Thus, using that $z_1(\mu_{2} + \tilde{k}_3)\neq
		z_2(\mu_1 + \tilde{k}_3)$ we we can reformulate the previous equation to obtain
	\begin{equation} \label{eq:remaining_coefficient_k2}
		\tilde k_2 = \frac{(z_2-z_1)(\mu_2+\tilde k_3)(\mu_1+\tilde k_3)}{z_1(\mu_2+\tilde k_3) - z_2(\mu_1 + \tilde k_3)}.
	\end{equation}
	Now, in equation \eqref{eq:remaining_coefficients_basic_equation} for $j=3$, replacing $\zeta \tilde K_1$ by the equality \eqref{eq:remaining_coefficients_reformulated_equation} for $j=2$ and plugging in the expression \eqref{eq:remaining_coefficient_k2} for $\tilde{k}_2$ we obtain, after some reformulations,
	\begin{multline} \label{eq:remaining_coefficients_k3_isolated}
		\tilde k_3 [(z_3-z_2)(\mu_2-\mu_1)z_1 - (z_2-z_1)(\mu_3-\mu_2)z_3] = \\ = (z_2-z_1)\mu_1[z_2\mu_3 - z_3 \mu_2] - (z_3-z_2)\mu_3[z_1\mu_2 - z_2 \mu_1].
	\end{multline}
	Using the definition of the $z_j$ in \eqref{eq:remaining_coefficients_basic_equation} we derive that the factor after $\tilde k_3$ in \eqref{eq:remaining_coefficients_k3_isolated} corresponds to the term
	\[
		K_1^2 k_2 \frac{(\mu_2-\mu_1)(\mu_3-\mu_2)(\mu_1-\mu_3)}{(k_2+k_3+\mu_1)(k_2+k_3+\mu_2)(k_2+k_3+\mu_3)}\neq 0
	\]
	which is nonzero by the $\mu_j$ being pairwise distinct. Thus, again plugging in the definition of the $z_j$ in \eqref{eq:remaining_coefficients_basic_equation} and rearranging the terms in \eqref{eq:remaining_coefficients_k3_isolated} yields $\tilde k_3 = k_3$ after some computations and, consequently, also $\tilde k_2 = k_2$ and $\zeta \tilde K_1 = K_1$ by the previous considerations.\\
	
	As a consequence,  the remaining $\zeta \tilde K_1^i,\tilde k_2^i,\tilde k_3^i$ considered in this final part of the proof are uniquely determined as $\zeta \tilde K_1^i = K_1^i, \tilde k_2^i = k_2^i$ and $\tilde k_3^i = k_3^i$
\end{proof}
The previous result shows that, already under knowledge of $C_\tis^i(t_l)$ for $i=1,\ldots,n$ and sufficiently many distinct time-points $t_l$, the coefficients $k_2^i,k_3^i$ and the coefficients $K_1^i$ can be determined uniquely and uniquely up to a constant, respectively. Considering the ODE system \eqref{eq:odesystem}, it is clear that this result cannot be improved in the sense that the constant factor of $K_1^i$ cannot be determined without any knowledge of $C_\art$ (since one can always divide all $K_1^i$ by a constant and multiply $C_\art$ by the same constant).

In case one aims to determine all parameters of a given configuration uniquely,
some additional measurements related to $C_\art$ are necessary. It is easy to
see that a single, non-zero measurement of $C_\art$, for instance, would
suffice. Indeed, given the value of a ground truth $\hat{C}_\art (\hat{s})\neq
	0$ at some time-point $\hat{s}$, the equality $C_\art(\hat{s}) = \hat{C}_\art
	(\hat{s}) = \tilde C_\art(\hat{s}) $ together with the result from Proposition
\ref{prop:main_uniqueness_result} immediately imply that $\zeta=1$ such that
all parameters are uniquely defined.

In current practice, indeed measurements of $C_\art$ are obtained via an
expensive blood-sample analysis, and used for parameter identification, see for
instance \cite{Ver13}. As discussed in the introduction, however, in contrast
to obtaining measurements of $C_\art$, it is much simpler to obtain
measurements of the total concentration $C_\tot$, where $C_\art = f C_\tot$
with the unknown parent plasma fraction $f$.

As the following result shows, measurements of $C_\tot$ only are indeed
sufficient to uniquely identify all remaining parameters, provided that one has
sufficiently many measurements in relation to a parametrization of $f$. To
formulate this, we need a notion of parametrization of the parent plasma
fractions.

\begin{definition}[Parametrized function class for parent plasma fractions] For any $q \in \N$, we say that a set of functions $F_q \subset \{f:\R \rightarrow \R\} $ is a degree-$q$ parametrized set if for any $f,\tilde{f}\in F_q$ and $\lambda \in \R$ it holds that $\lambda f - \tilde{f}$ attaining zero at $q$ distinct points implies that $\lambda=1$ and $f = \tilde{f}$.
\end{definition}
Simple examples of degree-$q$ parametrized sets of functions are polynomials of degree $q-1$ that satisfy $f(x_0) = c$ for some given $x_0,c\in \R $ with $c\neq 0$ or polyexponential functions of degree $q/2$ (if $q$ is even) that satisfy $f(x_0) = c$ for some given $x_0,c\in \R $ with $c\neq 0$. The latter is a frequently used type of parametrization for parent plasma functions (where $f(0)=1$ is required), see for instance \cite{Ver13}.

\begin{proposition} \label{prop:uniqueness_attenuation_maps}
	In the situation of Proposition \ref{prop:main_uniqueness_result}, assume in addition that $f,\tilde{f}:\R \rightarrow \R$ are
	parent plasma fractions contained in the same degree-$q$ parametrized set of functions, and are such that
	\[ C_\art(s_l) = f(s_l)C_\tot(s_l) \text{ and } \tilde C_\art(s_l) = \tilde f(s_l)C_\tot(s_l) \text{ for }l=1,\ldots,q,\]
	with $s_1,s_2,\ldots,s_q$ being $q$ different points, and $C_\tot(s_l)\neq 0$
	given for $l=1,\ldots,q$. Then, all assertions of Proposition
	\ref{prop:main_uniqueness_result} hold with $\zeta=1$, and further
	\[ f = \tilde{f} .\]
\end{proposition}
\begin{proof}
	Proposition \ref{prop:main_uniqueness_result} already implies that $\tilde{C}_\art = \zeta C_\art$. Using that, by assumption,
	\[   \zeta f(s_{l})C_\tot(s_{l})  = \zeta C_\art(s_{ l}) = \tilde C_\art(s_{l}) = \tilde f(s_{ l})C_\tot(s_{ l}),\]
	we obtain $(\zeta f - \tilde{f})(s_l) = 0$ for $l=1,\ldots,q$. Since
	$f,\tilde{f}:\R \rightarrow \R$ are parent plasma fractions contained in the
	same degree-$q$ parametrized set, this implies that $\zeta = 1$ and $f =
		\tilde{f}$ as claimed.
\end{proof}

The following theorem now summarizes results of the previous two propositions
in view of practical applications.
\begin{theorem} \label{thm:main_uniqueness} Let $(p,n,((\lambda_j,\mu_j))_{j=1}^p ,((K_1^i,k_2^i,k_3^i))_{i=1}^n,(C^i_\tis)_{i=1}^n,C_\art)$
	be a ground-truth \\ configuration of the irreversible two tissue compartment model such that
	\begin{enumerate}
		\item $p \geq 3$, $n\geq 3$ and  $K_1^i,k_2^i,k_3^i > 0$ for all $i=1,\ldots,n$,
		\item There are at least $p+3$ regions $i_1,\ldots,i_{p+3}$ where each the
		      $k_3^{i_s}$ and the $k_2^{i_s} + k_3^{i_s}$ are pairwise distinct for
		      $s=1,\ldots,p+3$.
	\end{enumerate}
	Let further be $C_\tot:[0,\infty) \rightarrow [0,\infty)$ be the ground truth total concentration.

	Then, for any other parameter configuration $(\tilde p,n,((\tilde
		\lambda_j,\tilde \mu_j))_{j=1}^{\tilde p} ,((\tilde K_1^i,\tilde k_2^i,\tilde
		k_3^i))_{i=1}^n,(\tilde C^i_\tis)_{i=1}^n,$ $\tilde C_\art)$ such that the
	conditions 1) and 2) above also hold, it follows from
	\[C_\tis(t_l) = \tilde C_\tis(t_l) \quad \text{for }l=1,\ldots, T\]
	with $T \geq \max\{ 2(p+3),2(\tilde p + 3)\}$ and the $t_1,\ldots,t_T$ pairwise
	distinct, that, for some constant $\zeta\neq 0$,
	\[  K_1^i = \zeta \tilde K_1^i,  k_2^i =\tilde k_2^i \text{ and }  k_3^i =\tilde k_3^i\text{ for all }i=1,\ldots,n,
	\]
	that $p= \tilde{p}$, and that (up to re-indexing)
	\[ \tilde \mu_j = \mu_j \text{ and } \tilde \lambda_j =\zeta  \lambda_j \text{ for all }i=1,\ldots,p.
	\]
	If further $f:[0,\infty) \rightarrow [0,\infty)$ is a ground-truth parent
	plasma fraction in a degree-$q$ \\ parametrized set of functions and $\tilde
		f:[0,\infty) \rightarrow [0,\infty)$ is a parent plasma fraction in the same
	degree-$q$ parametrized set of functions such that
	\[ C_\art(s_l) = f(s_l)C_{\tot}(s_l) \text{ and } \tilde C_\art(s_l) = \tilde f(s_l)C_{\tot}(s_l) \text{ for }l=1,\ldots,q,\]
	with the $s_1,\ldots,s_q$ pairwise distinct and $C_{\tot}(s_l)\neq 0$ given,
	then $\zeta=1$ and
	\[ f = \tilde{f}.
	\]
	\begin{proof}
		This is an immediate consequence of Lemma \ref{lem:simpler_version_of_main_assumption} and Proposition \ref{prop:main_uniqueness_result}: Indeed, Lemma \ref{lem:simpler_version_of_main_assumption} ensures that the assumptions of Proposition \ref{prop:main_uniqueness_result} are satisfied provided that 1.) and 2.) hold. In case $\tilde{p}\leq p$ the result immediately follows from Propositions \ref{prop:main_uniqueness_result} and \ref{prop:uniqueness_attenuation_maps}. In case $\tilde{p}> p$ it follows from interchanging the roles of the two configurations and again applying Propositions \ref{prop:main_uniqueness_result} and \ref{prop:uniqueness_attenuation_maps}.
	\end{proof}
\end{theorem}
\begin{remark}[Interpretation for practical application]
	Besides putting some basis assumptions on the ground truth-configuration and requiring positivity of the metabolic parameters, the previous theorem can be read as follows: If one obtains a configuration that matches the measured data, it can be guaranteed to coincide with ground-truth configuration if at least $\tilde{p}+3 $ of the found terms $\tilde k_2^i +  \tilde k_3^i$ and $\tilde k_3^i$ are pairwise distinct.

\end{remark}
\begin{remark}[Generalization for nontrivial fractional blood volume]
Following standard approaches in quantitative PET modeling, we assume here that the PET images provide exactly the tissue concentration $C_\tis$. A more realistic model would be that the voxel measurements provide a convex combination of the tissue and blood tracer concentration given by $C_\pet(t) = (1-\fbv)\cdot C_\tis(t)+\fbv\cdot C_\tot(t)$, where fbv with $0\leq \fbv\leq 0.05$ describes the fractional blood volume. In case the parameter $\fbv$ is known and $C_\tot$ is available at the same time points as the PET image measurements, our results cover also this setting. The general case, where both $\fbv$ and $C_\tot$ are unavailable, can be addressed by similar techniques as in the proof of Proposition \ref{prop:main_uniqueness_result}. Here, the idea would be to employ a polyexponential parametrization also for $C_\tot$, and assuming enough measurements of $C_\pet$ to be available in order to apply the unique interpolation result of Lemma \ref{lem:polyexpinter}. One would further have to ensure positivity of $C_\tot$, the intial condition $C_\tot(0)=0$ and conditions on the attenuation $f = C_\art/C_\tot$ such as monotonicity and limiting conditions with respect to time approaching zero and infinity, respectively. These requirements imply corresponding conditions on the parameters of $C_\art$ and $C_\tot$. 
\end{remark}

\section{A Tikhonov approach for parameter identification with noisy data}
\label{sec:tikhonov_parameter_id}
In the previous section we have established that, under appropriate conditions, the parameters $(K_1^i,k_2^i,k_3^i)$ of the irreversible two tissue compartment model
in regions $i=1,\ldots,n$ can be obtained uniquely from measurements $C_\tis(t_i)$, $i=1,\ldots,T$ and measurements $C_\tot(s_i)$ for $i=1,\ldots,q$. While this result shows that parameter identification is possible in principle, it considers the idealized scenario of exact measurements. In this section, we consider the situation of noisy measurements, for which we develop a Tikhonov approach for stable parameter identification. As main analytic results of this section, we show i) a stability result, i.e., that the proposed Tikhonov approach is stable with respect to (noise) variations in the measurements (see Theorem \ref{thm:well_posedness}) and ii) a consistency result, i.e., that in the limit of vanishing noise, solutions of the Tikhonov approach converge to the ground truth parameters (see Theorem \ref{thm:consistency}).

As first step, we define a forward model that maps the unknown parameters to
the available measurement data. To this aim, we define the arterial
concentration as mapping
\begin{equation} \label{eq:CA_parametriziation_tikhonov}
	\begin{aligned}
		C_\art:~  \mathbb{R}^p\times\mathbb{R}^p & \rightarrow \mathcal{P}_p\notag                                      \\
		\left(\lambda, \mu\right)                & \mapsto \left[ t \mapsto \sum_{i=1}^p \lambda_i e^{\mu_i t} \right].
	\end{aligned}
\end{equation}
Further, we define a parametrized parent plasma fraction as mapping
\begin{equation}
	\begin{aligned}
		\label{fglobal}
		f: & ~ \Mc \subset \R^{\hat{q}}\to F_q \\ %
		   & ~  m\mapsto f_m.
	\end{aligned}
\end{equation}
where $\Mc \subset \R^{\hat{q}}$ is some (finite dimensional) parameter space and $F_q$ is a degree-$q$ parametrized set of functions.

\begin{remark}[Parent plasma fraction example] \label{rem:attenuation_map_example} A classical model for the parent plasma fraction (see \cite{Ton15} for different models), that we will also use in our numerical experiments below, is the \emph{biexponential model}
	\[
		f\left(t\right) = Ae^{\xi_1 t}+\left(1-A\right) e^{\xi_2 t} ~ ~ ~ \text{ for } ~ ~ t\geq 0.
	\]
	Here \(\mathcal{M} = \left[0,\infty\right)\times\left(-\infty,0\right]^2\) and
	the degree of $F_q$ is $q=4$.
\end{remark}
In addition to the parameters of the functions modeling the arterial concentration and the parent plasma fraction, the forward model also includes the parameters $\Kb^i = (K_1^i,k_2^i,k_3^i)$ for $i=1,\ldots,n$ compartments. With this, the unknown parameters are summarized by $(\lambda,\mu,m,\Kb^i,\ldots,\Kb^n)$ and we denote by $X = \R^p \times \R^p \times \R^{\hat q} \times \R^{3 \times n}$ the resulting parameter space with norm
\[
	\Vert\left(\lambda, \mu, m, \textbf{K}^1, \dots, \textbf{K}^n\right)\Vert_X^2:=\sum_{j=1}^p\left(\vert\lambda_j\vert^2+\vert\mu_j\vert^2\right)+\Vert m\Vert_2^2+\sum_{i=1}^n\Vert\textbf{K}^i\Vert_2^2,
\]
where $\|\cdot \|_2$ denotes the Euclidean norm. Given measurement points
$t_1,\ldots,t_T$ for $C_\tis$ and $s_1,\ldots,s_q$ for the total concentration
$C_\tot$, those parameters are mapped forward to a measurement space $Y =
	\R^{n\times T+q}$, again equipped with the Euclidean norm $\|\cdot \|_Y = \|
	\cdot \|_2$ via the function
\begin{equation} \label{eq:forward_model_full}
	\begin{aligned}
		F: \mathcal{D}\left(F\right):=\mathbb{R}^p\times\mathbb{R}^p \times \mathcal{M}\times \left[\epsilon, \infty\right)^{3\times n}\subseteq X & \rightarrow Y           \\
		x                                                                                                                                          & \mapsto (F^1(x),F^2(x))
	\end{aligned}
\end{equation}
where, for $x = (\lambda,\mu,m, \mathbf{K}^1, \dots, \mathbf{K}^n)$
\begin{equation} \label{eq:forward_model_part_1}
	F^1(x) =
	\begin{pmatrix}
		C_\tis\left(C_\art\left(\lambda,\mu\right), \textbf{K}^1, \right)(t_1) & \dots  & C_\tis\left(C_\art\left(\lambda,\mu\right), \textbf{K}^1, \right)(t_T) \\
		\vdots                                                                 & \ddots & \vdots                                                                 \\
		C_\tis\left(C_\art\left(\lambda,\mu\right), \textbf{K}^n, \right)(t_1) & \dots  & C_\tis\left(C_\art\left(\lambda,\mu\right), \textbf{K}^n, \right)(t_T)
	\end{pmatrix}
\end{equation}
and
\begin{equation}\label{eq:forward_model_part_2}
	F^2(x) = 	 \begin{pmatrix}
		C_\tot\left(s_1\right)f_m\left(s_1\right)-C_\art\left(\lambda,\mu\right)\left(s_1\right) & \dots & C_\tot\left(s_q\right)f_m\left(s_q\right)-C_\art\left(\lambda,\mu\right)\left(s_q\right)
	\end{pmatrix}.
\end{equation}
Here $C_\tis\left(C_\art\left(\lambda,\mu\right), \textbf{K}^i,\right)$ denotes the solution of the irreversible two tissue compartment ODE model \eqref{eq:odesystem} with parameters $\Kb^i$ and arterial concentration $C_\art\left(\lambda,\mu\right)$.
Note that $F^2$ depends on the data $C_\tot$ that must be obtained from blood samples or PET measurements, which we assume to be given throughout this section. A further adaption of the model to include also $C_\tot$ as possibly noise measurement is possible with the same techniques as below, but will be omitted for the sake of simplicity.

Now denoting by $\hat{C}^i_\tis(t_1),\ldots,\hat{C}^i_\tis(t_T)$ for
$i=1,\ldots,n$ measurements corresponding to the ground-truth parameters, our
goal is to find parameters $(\lambda,\mu,m,\Kb^1,\ldots,\Kb^n)$ such that
\[ F(\lambda,\mu,m,\Kb^1,\ldots,\Kb^n) = \begin{pmatrix}
		\hat C_\tis^1\left(t_1\right) & \dots  & \hat C_\tis^1\left(t_T\right) \\
		\vdots                        & \ddots & \vdots                        \\
		\hat C_\tis^n\left(t_1\right) & \dots  & \hat C_\tis^n\left(t_T\right)
	\end{pmatrix} \times (0,\ldots,0) \in \R^{n \times T} \times \R^q.
\]
Accounting for the fact that the given parameters are perturbed by measurement
noise, i.e., we are actually given $(C^i_\tis)^\delta(t_l) $ with
\[ \sum_{i=1}^n \sum_{l=1}^T \|(C^i_\tis)^\delta(t_l) - \hat C^i_\tis(t_l) \|_2 ^2 \leq \delta ,\]
we address the parameter identification problem via a minimization problem of
the form
\begin{multline}
	\label{eq:tikhonov_main_min_prob}
	\min_{\left(\lambda,\mu,m, \mathbf{K}^1, \dots, \mathbf{K}^n\right)\in \mathcal{D}\left(F\right)} \Vert F\left(\lambda,\mu,m, \mathbf{K}^1, \dots, \mathbf{K}^n\right)- (C_\tis^\delta,0)\Vert_Y^2 \\ +\alpha\Vert \left(\lambda,\mu,m, \mathbf{K}^1, \dots, \mathbf{K}^n\right) - \left(\bar \lambda,\bar \mu,\bar m, \mathbf{\bar K}^1, \dots, \mathbf{\bar K}^n\right) \Vert_X^2.
\end{multline}
Here $0 \in \R^q$ is a $q$-dimensional vector of zeros, $C_\tis^\delta$ summarizes the available measurements for $C_\tis^\delta$, i.e.,
\[  C^\delta _\tis = \begin{pmatrix}
		(C_\tis^1)^\delta\left(t_1\right) & \dots  & ( C^1_\tis)^\delta\left(t_T\right) \\
		\vdots                            & \ddots & \vdots                             \\
		(C^n_\tis)^\delta\left(t_1\right) & \dots  & (C_\tis^n)^\delta\left(t_T\right)
	\end{pmatrix}.
\]
and $\left(\bar \lambda,\bar \mu,\bar m, \mathbf{\bar K}^1, \dots, \mathbf{\bar
		K}^n\right) $ is an initial guess on the ground truth parameters. The above
approach corresponds to \emph{Nonlinear Tikhonov-Regularisation}, for which
stability and consistency results can be ensured as follows.

\begin{theorem}[Well-posedness and stability] \label{thm:well_posedness}
	Let the parent plasma fractions $f_m$ be such that the mapping $m \mapsto f_m(t)$ is continuous for any $t \in [0,\infty) $. Then, for any given datum $C_\tis^\delta$, the minimization problem \eqref{eq:tikhonov_main_min_prob} admits a solution. Moreover, solutions are stable in the sense that, if $(C_\tis^{\delta_k})_k$ is a sequence of data converging to some datum $C_\tis^{\delta}$, then, any sequence of solutions $(x^k)_k$ of \eqref{eq:tikhonov_main_min_prob} with data $(C_\tis^{\delta_k})_k$ admits a convergent subsequence, and the limit of any convergent subsequence $x$ is a solution of \eqref{eq:tikhonov_main_min_prob} with data  $C_\tis^{\delta}$.
	\begin{proof}
		Since $X$ and $Y$ are finite dimensional and $D(F)$ is obviously closed, this follows from classical results in regularization theory, see for instance \cite[Theorem 10.2]{Engl2000} provided that $F$ is continuous.

		We start with continuity of $F^1$ as in \eqref{eq:forward_model_part_1}, the
		first component of $F$. For this, it suffices to show that the mapping from the
		the parameter $(\lambda,\mu, \mathbf{K}^1, \dots, \mathbf{K}^n)$ to
		$C_\tis^i(t)$, with $t \in [0,\infty)$ fixed, is continuous, which, in turn,
		follows from the representation of $C_\tis^i(t)$ as in
		\eqref{eq:ct_representation_general} if, for any $g \in L^2(0,t)$ and any
		sequence $(\lambda^l,\mu^l)_l$ converging to $(\lambda,\mu)$ it holds that
		\begin{equation}
			\label{pweakconv}
			\int_0^tg\left(s\right)\left(C_\art\left(\lambda^l,\mu^l\right)-C_\art\left(\lambda,\mu\right)\right)\left(s\right)\dx s \to 0 ~ ~ ~ \text{ as } ~ ~ l\to\infty.
		\end{equation}
		By Hölder's inequality, the latter follows from \(C_\art\left(\lambda^l,\mu^l\right)\to C_\art\left(\lambda, \mu\right)\) in \(L^2\left(0,T_{\text{max}}\right)\), which, in turn, follows via the Lebesgue dominated convergence theorem from point-wise convergence of $C_\art\left(\lambda^l,\mu^l\right)$ and the fact that $|C_\art\left(\lambda^l,\mu^l\right)|$ on $[0,t]$ can easily be bounded by a constant independent of $l$.

		Regarding $F^2$ as in \eqref{eq:forward_model_part_2}, the second component of
		$F$, continuity immediately follows from continuity of $(\lambda,\mu) \mapsto
			C_\art(\lambda,\mu)(t)$ and $m \mapsto f_m(t)$ for any $t \in [0,\infty)$
		fixed, where the latter holds by assumption.
	\end{proof}
\end{theorem}
\begin{remark}[Continuity of $m\mapsto f_m(t)$]
	Note that the assumption of continuity of $m\mapsto f_m(t)$ is only necessary since we allow for arbitrarily parametrized parent plasma fractions; it holds in particular for the biexponential model of Remark \ref{rem:attenuation_map_example} and will typically hold for any reasonable parametrization.
\end{remark}

At last in this section we now establish a consistency result.
\begin{theorem}[Consistency] \label{thm:consistency}
	Let $(\hat p,n,((\hat \lambda_j,\hat \mu_j))_{j=1}^{\hat p} ,((\hat K_1^i,\hat k_2^i,\hat k_3^i))_{i=1}^n,(\hat C^i_\tis)_{i=1}^n,\hat C_\art)$ be a ground-truth configuration of the irreversible two tissue compartment model
	satisfying the assumptions of Theorem \ref{thm:main_uniqueness}, and let $f_{\hat m}\in \Mc$ be a ground-truth
	parent plasma fraction.

	With $\hat x = (\hat \lambda,\hat \mu,\hat m, \mathbf{\hat K}^1, \dots,
		\mathbf{\hat K}^n)$ the corresponding parameters and $\hat y := F(\hat{x})$ the
	corresponding measurement data, let $y^{\delta_k}$ be any sequence of noisy
	data such that $\|\hat{y}-y^{\delta_k} \| \leq \delta_k$ with $\delta_k >0$,
	$\lim_{k\rightarrow \infty} \delta_k = 0$.

	Then, any sequence of solutions $(x_k)_k$ of \eqref{eq:tikhonov_main_min_prob}
	with data $y^\delta = y^{\delta_k}$ and $\alpha = \alpha_k$ such that $\alpha_k
		\rightarrow 0$ and $\delta_k^2/\alpha_k \rightarrow 0$ as $k \rightarrow 0$
	admits a convergent subsequence. Any limit $ x = ( \lambda, \mu, m, \mathbf{
			K}^1, \dots, \mathbf{ K}^n)$ of such a subsequence, such that the corresponding
	parameter configuration satisfies the assumptions of Theorem
	\ref{thm:main_uniqueness}, coincides with $\hat{x}$. Further, if any limit of a
	convergent subsequence corresponds to a parameter configuration satisfying the
	assumptions of Theorem \ref{thm:main_uniqueness}, then the entire sequence
	$(x_k)_k$ converges to $\hat{x}$.
	\begin{proof}
		This is a consequence of Theorem \ref{thm:main_uniqueness}, which ensures that there is a unique $x \in X$ with $F(x) = \hat y$, and of classical results from regularization theory, see for instance \cite[Theorem 10.3]{Engl2000}.
	\end{proof}
\end{theorem}

\begin{remark}[Interpretation of the consistency result] When choosing $p\geq 3$ and $n\geq 3$, and given the definition of $\Dc(F)$ as in \eqref{eq:forward_model_full}, the above consistency result together with the unique reconstructability result of Theorem \ref{thm:main_uniqueness} can be interpreted as follows: Whenever the parameters $(K_1^i,k_2^i,k_3^i)_{i=1}^n$ corresponding to a limit $x$ of $(x_k)_k$ are such that at least $p+3$ of the parameters $k^i_3$ and $k^i_2 + k_3^i$ are pairwise distinct, then one can ensure that $x = \hat{x}$.
\end{remark}

\begin{remark}[Multi-parameter regularization]
	The setting of \eqref{eq:tikhonov_main_min_prob} and the subsequent results on well-posedness and consistency can be generalized to incorporating different regularization parameters for the different norms and data terms, see for instance \cite{holler18coupled_mh}, which is reasonable given the fact that the parameters might live on different scales, and given the fact that the noise level of different measurements over time might be different.
\end{remark}

\begin{remark}[Model Variations]
	Currently, in the setting of \eqref{eq:tikhonov_main_min_prob}, the parameters $(K_1^i,k_2^i,$ $k_3^i)_{i=1}^n$ are bounded away from zero by $\epsilon>0$. For the $(\mu_j)_{j=1}^p$, we currently do not pose any constraints even though, as mentioned in Remark \ref{rem:sign_of_exponents}, only the choice $\mu_j < 0$ is reasonable from a physiological perspective. Likewise, $C_\art$ as parametrized in \eqref{eq:CA_parametriziation_tikhonov}, does not necessarily satisfy $C_\art(0) = 0$. These two conditions, however, can be easily incorporated in the model via the additional constraint $\mu_j \leq -\tilde{\epsilon}$ for some $\tilde{\epsilon} \geq 0$ and via setting $\lambda_p = -\sum_{j=1}^{p-1} \lambda _j$, respectively.
\end{remark}

\section{Numerical solution algorithm} \label{sec:algorithm}

In this section, we provide proof-of-concept numerical experiments that
illustrate the analytic unique identifiability results of Sections
\ref{sec:unique_identifiability} and \ref{sec:tikhonov_parameter_id} also
numerically.

\subsection{Experimental setup}
We consider the following experimental setup: As ground truth data, we consider
$n=3$ different anatomical regions, where, based on the realistic values
provided in \cite[Table 1]{Dim91}, the parameters are chosen as
\begin{equation}
	\label{metabolicp}
	\begin{pmatrix}K_1^1\\ k_2^1\\ k_3^1\end{pmatrix} = \begin{pmatrix}0.157\\ 0.174\\ 0.118\end{pmatrix}, ~ \begin{pmatrix}K_1^2\\ k_2^2\\ k_3^2\end{pmatrix} = \begin{pmatrix}0.161\\ 0.179\\ 0.096\end{pmatrix} ~ ~ \text{ and } ~ ~ \begin{pmatrix}K_1^3\\ k_2^3\\ k_3^3\end{pmatrix} = \begin{pmatrix}0.177\\ 0.159\\ 0.088\end{pmatrix}.
\end{equation}
Here, region 1 corresponds to the frontal cortex, region 2 to the temporal cortex and region 3 to the occipital cortex in the human brain. We model the true arterial concentration by the triexponential function given by
\[
	C_\art\left(t\right) = -5 e^{-0.5 t}+4 e^{-0.2 t}+e^{-0.1t}
\]
and the parent plasma fraction for the biexponential model by
\[
	f\left(t\right) = 0.1 e^{-0.005 t}+0.9e^{-0.1 t}
\]
for \(0\leq t\leq T_{\text{max}}\) where \(T_{\text{max}}=3750 ~\text{sec} ~ ~
\widehat{=} ~ ~ 62.5 ~ \text{min}\). See Figure \ref{fig:evolution} for the
visualisation of the parent plasma fraction, the arterial concentration and the
measured total and additive concentrations.
\begin{figure}[t]
	\begin{center}
		\includegraphics[scale=0.33]{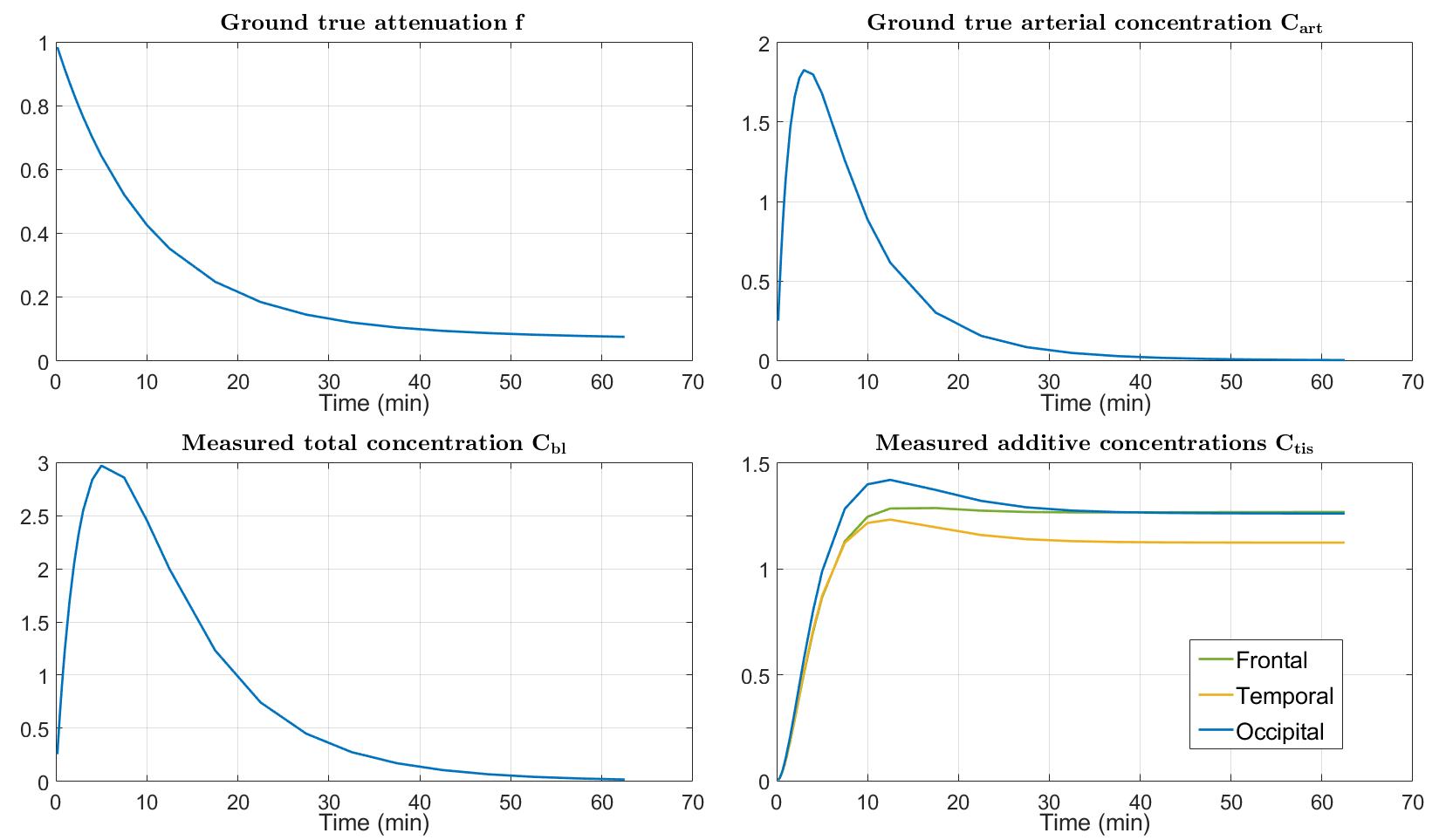}
		\caption[Ground true evolutions and measurements]{The evolutions in the first two plots display the ground true
			parent plasma fraction and arterial concentration. The third and fourth plot correspond to the resulting simulated measurements of the total and additive concentrations depending on the
			parent plasma fraction and the arterial concentration.}\label{fig:evolution}
	\end{center}
\end{figure}
We assume that we conduct a PET examination for \(T=25\) time frames where, equidistantly distributed, six take place in the first minutes, four in the following two minutes, two in the next two minutes, three in the next \(7.5\) minutes and finally ten in the remaining \(50\) minutes. Regarding $C_\tot$, we assume that measurements $C_\tot(s_1),\ldots,C_\tot(s_q)$ are available at the same timepoints $0=t_1,\ldots,t_T$ at which measurements of $C_\tis$ are available, i.e., $q=T$. This reflects the situation that an estimate of the total arterial tracer concentration is obtained from the PET images (rather than blood sampling), a technique that for which many recent works exist \cite{Zanotti-Fregonara2011}.
In view of Propositions \ref{prop:main_uniqueness_result} and \ref{prop:uniqueness_attenuation_maps}, the above experimental setting satisfies the assumptions such that unique identifiability from noiseless data can be guaranteed.

We summarize the unknown parameters of this setting by
\[
	x^\dagger = \left(\lambda_1, \lambda_2, \lambda_3, \mu_1, \mu_2, \mu_3, m_1, m_2, m_3, K_1^1, k_2^1, k_3^1, K_1^2, k_2^2, k_3^2, K_1^3, k_2^3, k_3^3\right)^T \in \R^{18},
\]
where $x^\dagger$ denotes the ground-truth parameters as specified above. For a
given number of measurements $T \in \N$, the data is summarized in vectorized
form via
\[ y^\dagger = (y_1,0) = F(x^\dagger) \in \R^{3T+T},
\]
where, abusing notation, $F:\R^{18} \to \R^{3T+T}$ is a vectorized version
of the forward model of \eqref{eq:forward_model_full}.

For the subsequent numerical experiments, we will also consider noisy versions
of the data, denoted by $y^{\delta_y}$ with $\delta_y$ being the noise level.
Those are defined by adding Gaussian noise with zero mean and variance
\(\frac{\delta_y^2}{3T}\) to each of the first $3T$ entries of $y^\dagger$
(note that no noise is added to the zero-entries), i.e., $(y^\delta)_l =
	y^\dagger _l + \eta _l$ with $\eta_l \sim
	\mathcal{N}\left(0,\frac{\delta_y^2}{3T}\right) $ for $l=1,\ldots,3T$, such
that
\[
	\mathbb{E}\left(\Vert y^\delta-y^\dagger\Vert_Y^2\right) = \sum_{l=1}^{3T}\mathbb{E}\left(\eta_{l}^2\right) = \sum_{l=1}^{3T}\frac{\delta_y^2}{3T} = \delta_y^2.
\]
where $\mathbb{E}(\cdot) $ denotes the expectation.

As we are dealing with locally convergent methods, it will also be important to
choose a reasonable initial guess $x_0$ for the algorithm. In order to test the
performance of the algorithm in dependence on how close the initial guess is to
the true solution, we employ the following steps to obtain perturbed initial
guesses. Given a level of perturbation $\delta_x$, we define
\begin{equation} \label{eq:x0_perturbations}
	x_0 = x^\dagger (1+\sigma \gamma)
\end{equation}
where \(\sigma \sim \text{Unif}\left(\left\{-1,1\right\}\right)\), i.e., is uniformly distributed on $\{-1,1\}$ and \(\gamma \sim \mathcal{N}\left(\delta_x,\frac{1}{4}\delta_x\right)\), i.e., is Gaussian distributed with mean $\delta_x$ and variance $\delta_x/4$. This results in a expected squared deviation of $x_0$ from $x^\dagger $ as by
\begin{equation}
	\label{Eestimation}
	\mathbb{E}\left(\frac{\Vert x_0-x^\dagger\Vert_X^2}{\Vert x^\dagger\Vert_X^2}\right) =\frac{1}{\Vert x^\dagger\Vert_X^2}\sum_{i=1}^{18}\left(x^\dagger_i\right)^2\mathbb{E}\left(\sigma_i^2 \gamma_i^2 \right) = \mathbb{E}\left(\gamma_1^2\right) = \frac{1}{4}\delta_x+\delta_x^2,
\end{equation}
where $\sigma_i\sim \text{Unif}\left(\left\{-1,1\right\}\right) $ and  $\gamma_i \sim \mathcal{N}\left(\delta_x,\frac{1}{4}\delta_x\right) $ for $i=1,\ldots,18$ are independent random variables.

\subsection{Algorithmic implementation}
In order to numerically solve the non-linear parameter identification, we
employ the \emph{iteratively regularized Gauss-Newton method} of \cite{Bak92},
see also \cite[Section 11.2]{Engl2000}. This is a standard method for solving
non-linear inverse problems. It is related to the Tikhonov approach discussed
in Section \ref{sec:tikhonov_parameter_id} in the sense that similar results on
stability and convergence/consistency (under appropriate source conditions) can
be obtained, see for instance \cite{Bla97,Hoh97}, but different to the Tikhonov
approach, regularization is achieved by early stopping of the algorithm rather
than adding an additional penalty term to the data-fidelity term. Early
stopping has the advantage that, using an estimate of the noise level of the
data, the discrepancy principle \cite[Section 4.3]{Engl2000} can be used to
determine the appropriate amount of regularization, without requiring multiple
solutions of a minimization problem as would be the case with the Tikhonov
approach.

Given an initial guess $x_0 \in \Dc(F)$ and a sequence of regularization
parameters $(\alpha_k)_k$ such that
\begin{equation}
	\label{eq:alphacond}
	\alpha_k>0, ~ ~ 1\leq \frac{\alpha_k}{\alpha_{k+1}}\leq c_\alpha, ~ ~ \lim_{k\to\infty}\alpha_k=0,
\end{equation}
where \(c_\alpha>1\) is some constant, the iteration steps of the iteratively regularized Gauss-Newton method for $k=0,1,2,\ldots$ are given as
\begin{equation}
	\label{eq:irgnm_iteration}
	x_{k+1}^\delta=x_k^\delta+\left(F'\left[x_k^\delta\right]^TF'\left[x_k^\delta\right]+\alpha_kI\right)^{-1}\left(F'\left[x_k^\delta\right]^T\left(y^\delta-F\left(x_k^\delta\right)\right)+\alpha_k\left(x_0-x_k^\delta\right)\right),
\end{equation}
where $F'[x_k^\delta] \in \R^{(nT+q)\times (2p+3+2n)}$ denotes the Jacobian Matrix of $F$ at $x_k^\delta$ and $F'\left[x_k^\delta\right]^T$ its transpose.

Those iteration steps are repeated until the \emph{discrepancy principle} is
satisfied, that is, until $\Vert F\left(x_k^\delta\right)-y^\delta\Vert_Y\leq
	\tau\delta$ holds for the first time, where $\delta $ is an estimate of
$\|y^\delta - y^\dagger \|_Y$ and $\tau >1$ with $\tau \approx 1$ is a
parameter. The iterate $x_k$ is then returned as the approximate solution of
$F(x) \approx y^\delta$.

\begin{remark}[Guaranteed convergence] Since the parameter identification problem addressed here is highly non-linear, global convergence guarantees for any numerical solution algorithm are out of reach. For the iteratively regularized Gauss-Newton method together with the discrepancy principle, as considered here, at least local convergence guarantees can be obtained as long as a particular source condition, i.e., a regularity condition on the ground truth solution holds, see \cite{Hoh97} for details.
\end{remark}

In a practical application, the iteration \eqref{eq:irgnm_iteration} is
combined with a projection on $\Dc(F)$, which is a closed, convex set for which
the projection is explicit (we denote the projection map by
$\mathcal{P}_{\Dc(F)}$), see \cite[Theorem 4]{Ka06} for corresponding results
on convergence of such a projected method. Together with this, we arrive at the
algorithm for solving $F(x) \approx y^\delta$ as provided in Algorithm
\ref{alg:irgnm}, where we set $\epsilon = 10^{-3}$ for defining $\Dc(F) = \R^3
	\times \R^3 \times [0,\infty) \times (-\infty,0]^2 \times
	[\epsilon,\infty)^{3T}$.

\begin{algorithm}
	\caption{Parameter Identification by IRGNM}\label{alg:irgnm}
	\begin{algorithmic}
		\STATE{\textbf{Input: }$\delta_x, \delta_y>0, ~~ \tau>1, ~~(\alpha_i)_i,$}
		\STATE{\hspace{1.1cm} $x_0\in\mathcal{D}\left(F\right),~~ y^\delta \in Y, (C_\tot(t_j))_{j=1}^T$}

		\STATE{\textbf{Initialise:} $r_0 \gets y^\delta- F\left(x_0\right)$}
		\STATE{\hspace{2cm}$ i \gets 0$}
		\WHILE{$\Vert r_i\Vert_Y>\tau\delta_y$}
		\STATE{$\mathcal{A} \gets F'\left[x_i\right]^*$}
		\STATE{$\mathcal{B} \gets \mathcal{A}\mathcal{A}^*$}
		\STATE{\textit{Solve   } $~~\mathcal{B}\left(x-x_i\right) = \mathcal{A}r_i+\alpha_i\left(x_0-x_i\right)~~$ \textit{for} $~~x\in X$}
		\STATE{$x_{i+1} \gets \mathcal{P}_{\mathcal{D}\left(F\right)}\left(x\right)$}%
		\STATE{$r_{i+1} \gets y^\delta-F\left(x_{i+1}\right)$}
		\STATE{$i\gets i+1$}
		\ENDWHILE
		\RETURN $x_k$
	\end{algorithmic}
\end{algorithm}
For the regularization parameters \(\left(\alpha_i\right)_i\) we choose the ansatz
\begin{equation}
	\label{alpharegi}
	\alpha_i = a e^{-bi}
\end{equation}
for \(i \in \mathbb{N}\) where $a=800 $ and $b = 1/5$ are fixed parameters. Besides fulfilling the decay conditions of \eqref{eq:alphacond}, this choice is motivated by the goal of penalizing deviations from the initial guess rather strongly at early iterations ($a$ large), and avoiding an exploding condition number of the matrix $\left(F'\left[x_k^\delta\right]^TF'\left[x_k^\delta\right]+\alpha_kI\right)$, that needs to be inverted at each iteration, during later iterations ($b$ rather small).

For the realization of the forward operator \(F\) and the adjoint of its
Fréchet-Differential the main idea is to vectorise the computations and omit
expensive for-loops. For that, one may exploit that the entries of
\(F\left(x\right)\) and \(F'\left[x\right]^*\), which mostly consist of
integral type entities, may be computed analytically. The elementary components
are of the form \(\int e^{\mu s}\dx s$, $ \int
e^{\left(k_2+k_3\right)\left(s-t\right)}e^{\mu s}\dx s\)$, $\(\int
se^{\left(k_2+k_3\right)\left(s-t\right)}e^{\mu s}\dx s\) and \(\int se^{\mu
		s}\dx s\). The latter may be computed by hand applying integration by parts.
The corresponding terms, which depend on a combination of time evaluations,
region and polyexponential parameters of \(C_\art\), are saved in
three-dimensional tensors which are overloaded throughout a respective
iteration to finally build up the adjoint operator of the Fréchet-Differential.
For the implementation of the IRGNM we will use the computational software
\textsc{Matlab} (see \cite{MATLAB:2020}).

\section{Experimental results}
\label{sec:experiments}
In order to evaluate the identifiability of the parameters $x^\dagger$ from noisy data $y^\delta \approx F(x^\dagger)$, we consider both the situation of noiseless data ($\delta_y = 0)$ and different noise levels $\delta_y \in  \{10^{-4},10^{-3},10^{-2}\} $. For the latter, it is important to note that, relative to the magnitude of the data, the noise level $\delta_y = 10^{-2}$ is already rather high. Indeed, as can be observed in Figure \ref{fig:magnit}, which depicts the average magnitudes of different regions in the data, about $25\%$ of the measurements have a magnitude below $10^{-1}$. For those values, adding Gaussian noise with standard deviation $\delta_y=10^{-2}$, for instance, results in noisy data whose standard deviation is already between $10\%$ and $100\%$ of the magnitude of the data itself.
\begin{figure}[t]
	\begin{center}
		\includegraphics[scale=0.25]{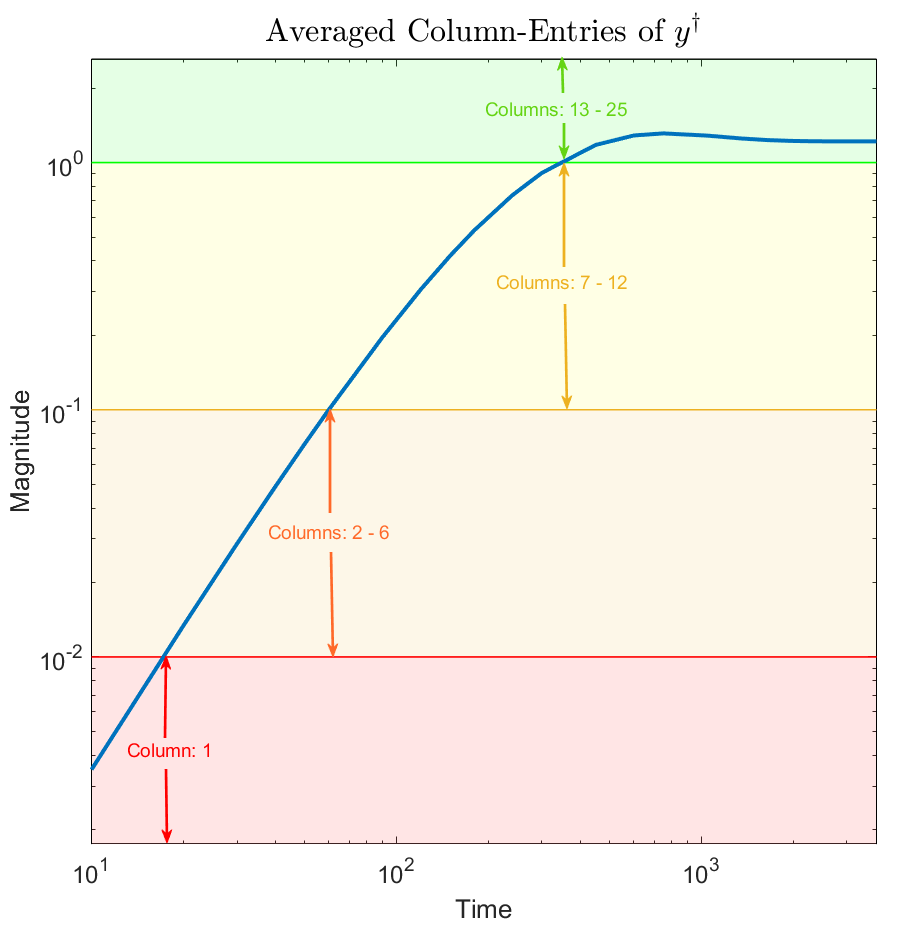}
		\caption{\label{fig:magnit} Visualization of magnitudes occurring in different time-intervals (i.e., columns) of the data $y^\dagger$}
	\end{center}
\end{figure}

For defining the initialization of the algorithm, we test with perturbations of
the ground truth as defined in \eqref{eq:x0_perturbations} for $\delta_x \in
	\{0.01,0.05,0.1,0.15\}$. Recall that those are relative perturbations such that
the square root of the expected squared deviation from the ground truth
$x^\dagger$ is between $\approx 5\%$ for $\delta_x = 0.01$ and $\approx 25\%$
for $\delta_x = 0.15$.

To quantify the improvement compared to the initialization that is obtained by
the algorithm, in addition to plotting the residual value $\|F(x_k) -
	y^\delta\|_Y$ and the relative error $\frac{\|x_k - x^\dagger \|_X}{\|x^\dagger
		\|_X}$ over iterations, we provide the following two values:
\[\rho_{opt} = 100 \left( 1 - \Vert x_0^\delta -x^\dagger\Vert_X^{-1} \min_{1\leq k\leq iter_{\text{max}}} \Vert x_k^\delta -x^\dagger\Vert_X\right)
\]
provides the best possible improvement (in percent, relative to the
initialization) that was obtained during all iterations, and
\[
	\rho_d = 100 \left( 1 - \Vert x_0^\delta -x^\dagger\Vert_X^{-1} \Vert x_N^\delta -x^\dagger\Vert_X\right) \%
\]
provides the improvement that was obtained at iteration $N$ where the algorithm
was stopped by the discrepancy principle, i.e., the first iteration $N$ where $
	\Vert F\left(x_N^\delta\right)-y^\delta\Vert_Y\leq \tau \delta$ was fulfilled.

Experiments where carried out for each combination of $\delta_y$ and $\delta_x$
as above. In order to obtain representative results, each experiment was
carried out 100 times. Those experiments where the algorithm diverged (i.e., no
improvement compared to the initialization was achieved) were dropped (see
Table \ref{tab:irgnm} for the number of dropped experiments for each parameter
combination) and, among the remaining ones, the one whose performance was
closest do the median performance of all repetitions was selected for the
figures below.

\begin{table}[t]
	\caption{Number (full setting/setting with known $C_\art$) of experiments (out of ten) not included in the final evaluation due to divergence. \label{tab:irgnm}}
	\begin{center}
		\begin{tabular}{ccccc}
			                                & $\boldsymbol{\delta_x=0.15}$ & $\boldsymbol{\delta_x=0.1}$ & $\boldsymbol{\delta_x=0.05}$ & $\boldsymbol{\delta_x=0.01}$ \\ \toprule
			$\boldsymbol{\delta_y=0}$       & 0/0                          & 0/0                         & 0/0                          & 0/0                          \\
			$\boldsymbol{\delta_y=10^{-4}}$ & 77/36                        & 65/17                       & 30/2                         & 4/0                          \\
			$\boldsymbol{\delta_y=10^{-3}}$ & 58/21                        & 40/8                        & 18/1                         & 2/0                          \\
			$\boldsymbol{\delta_y=10^{-2}}$ & 32/14                        & 29/4                        & 13/1                         & 2/0                          \\  \bottomrule
		\end{tabular}
	\end{center}

\end{table}

Figures \ref{fig:irgnm_del_0} to \ref{fig:irgnm_del_c2} show the result
obtained for different choices of $\delta_y = 0,10^{-4},10^{-3},10^{-2}$. In
those figures, the top and bottom rows depict the evaluation of the residual
value $\|F(x_k) - y^\delta\|_Y$ and the relative error $\frac{\|x_k - x^\dagger
		\|_X}{\|x^\dagger \|_X}$ over iterations, respectively, both with a logarithmic
scale for the vertical axis. The different columns show results for the
different choices of $\delta_x = 0.15,0.1,0.05,0.01$. The values of
$\rho_{opt}$ and $\rho_d$ (the latter only for $\delta_y>0$), together with the
respective iterations, are provided at the top of plots of the second row. Note
that, while the algorithm was always run until a fixed, maximal number of
iterations (300 for $\delta_y =0$ and $200$ for $\delta_y=0$) for obtaining the
figures, in practice the iterations would be stopped by the discrepancy
principle for $\delta_y>0$. For $\delta_y>0$, the red lines always indicate the
levels of the residual value (top row) and iteration number (bottom row) where
the algorithm would have stopped according to the discrepancy principle.

In Figure \ref{fig:irgnm_del_0}, which considers the noiseless case, one can
observe that the ground truth parameters $x^\dagger$ are approximated very well
for all levels of $\delta_x$. This confirms our analytic unique identifiability
result also in practice.

Considering Figures \ref{fig:irgnm_del_c4} to \ref{fig:irgnm_del_c2}, one can
observe that, in the low noise regime, a good approximation of the ground truth
$x^\dagger$ is still possible across different values of $\delta_x$. For higher
noise, a good initialization (i.e., a small value of $\delta_x$) is of
increasing importance and, in some cases, in particular for $\delta_y =
	10^{-2}$ and the parameters of the parent plasma fraction $f$, the ground truth
can not be recovered reasonably well.

\begin{figure}
	\begin{center}
		\includegraphics[scale=0.25, trim=2cm 0 4cm 0, clip]{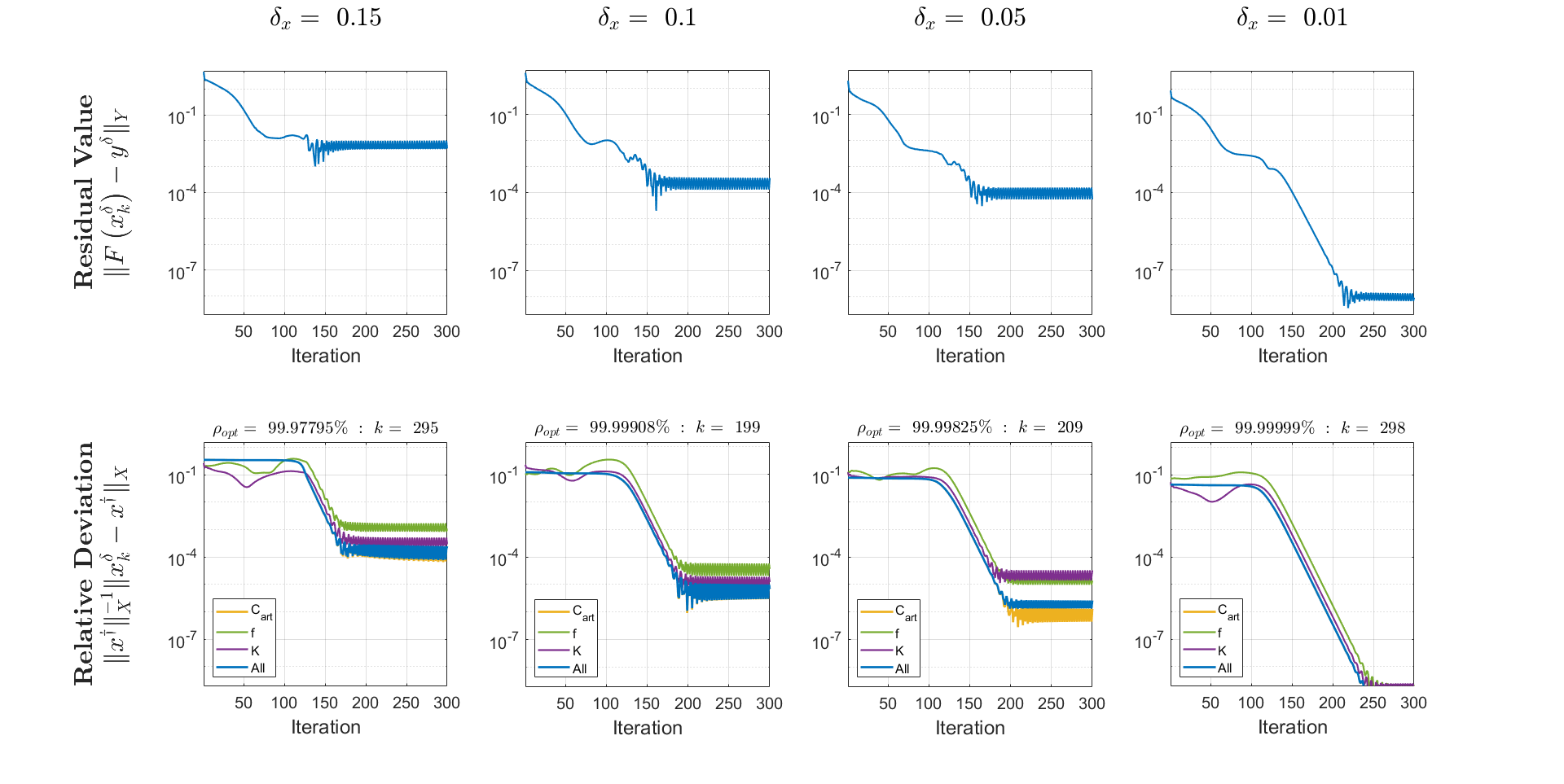}
		\caption[Performance of IRGNM for \(\delta_y =0\) for different initial guesses]{Performance of IRGNM for \(\delta_y =0\) and different initial guesses}\label{fig:irgnm_del_0}
	\end{center}
\end{figure}

\begin{figure}
	\begin{center}
		\includegraphics[scale=0.25, trim=2cm 0 4cm 0, clip]{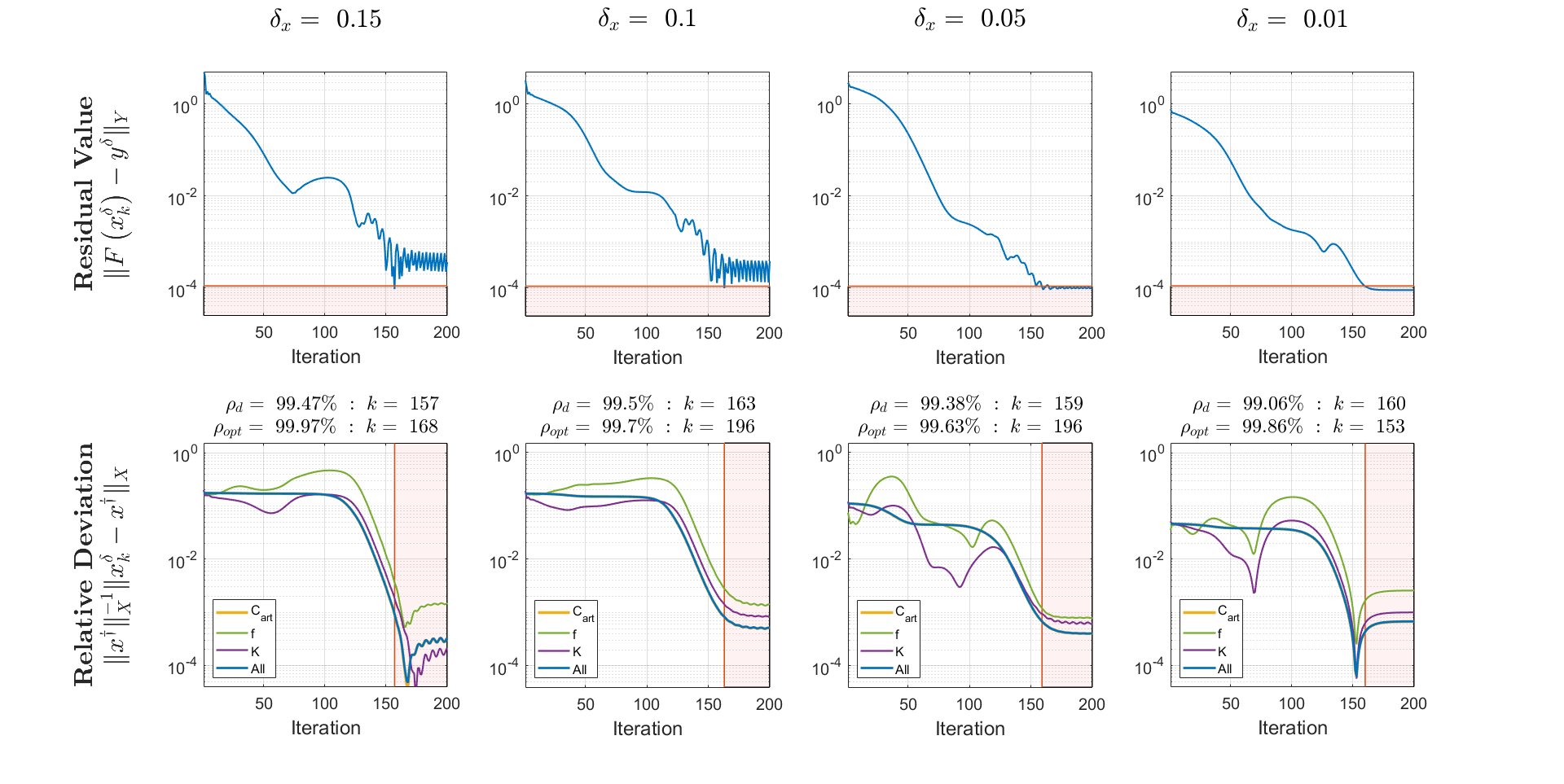}
		\caption[Performance of IRGNM for \(\delta_y =10^{-4}\) for different initial guesses]{Performance of IRGNM for \(\delta_y =10^{-4}\) and different initial guesses}\label{fig:irgnm_del_c4}
	\end{center}
\end{figure}
\begin{figure}
	\begin{center}
		\includegraphics[scale=0.25, trim=2cm 0 4cm 0, clip]{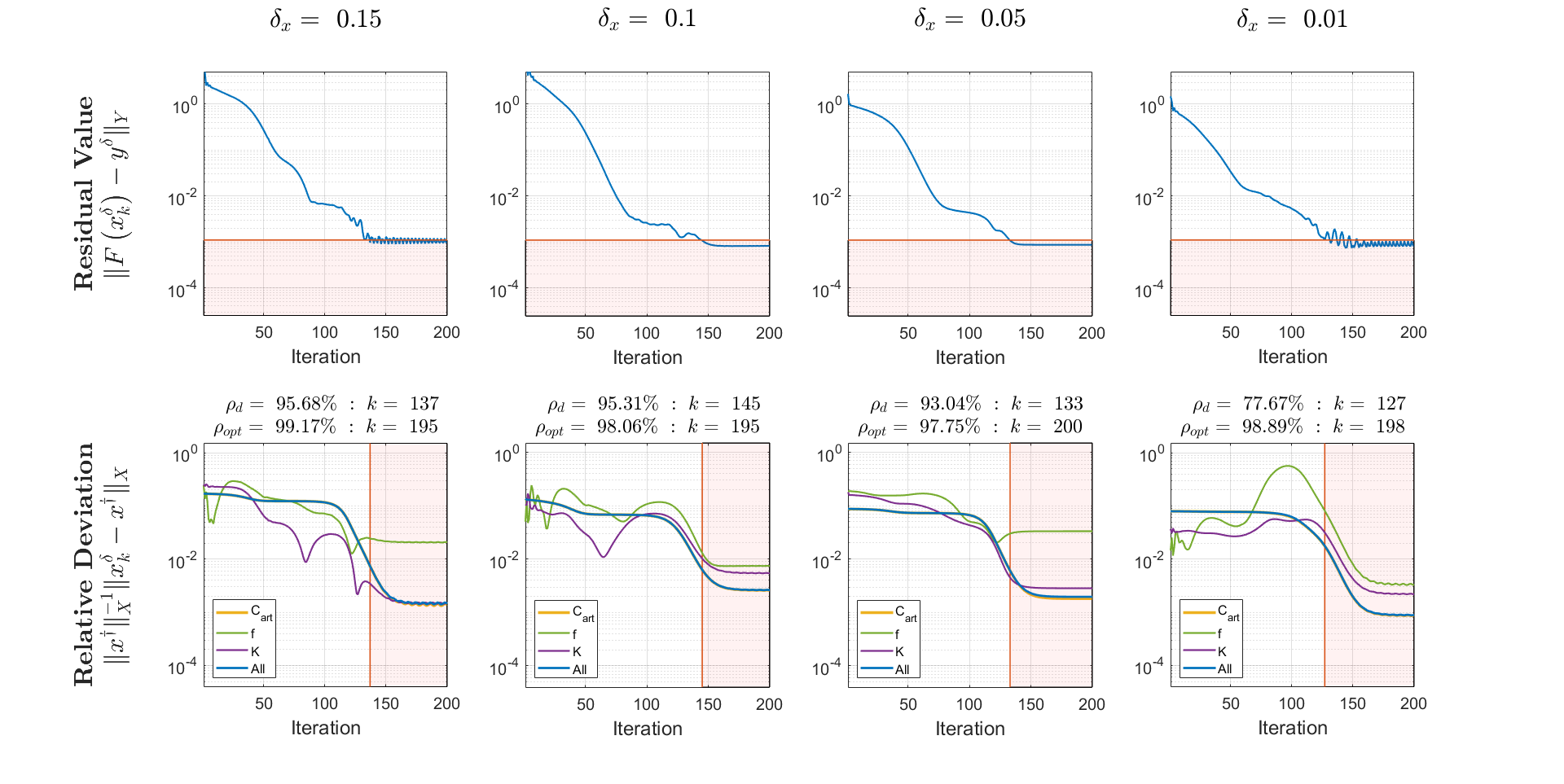}
		\caption[Performance of IRGNM for \(\delta_y =10^{-3}\) for different initial guesses]{Performance of IRGNM for \(\delta_y =10^{-3}\) and different initial guesses}\label{fig:irgnm_del_c3}
	\end{center}
\end{figure}
\begin{figure}
	\begin{center}
		\includegraphics[scale=0.25, trim=2cm 0 4cm 0, clip]{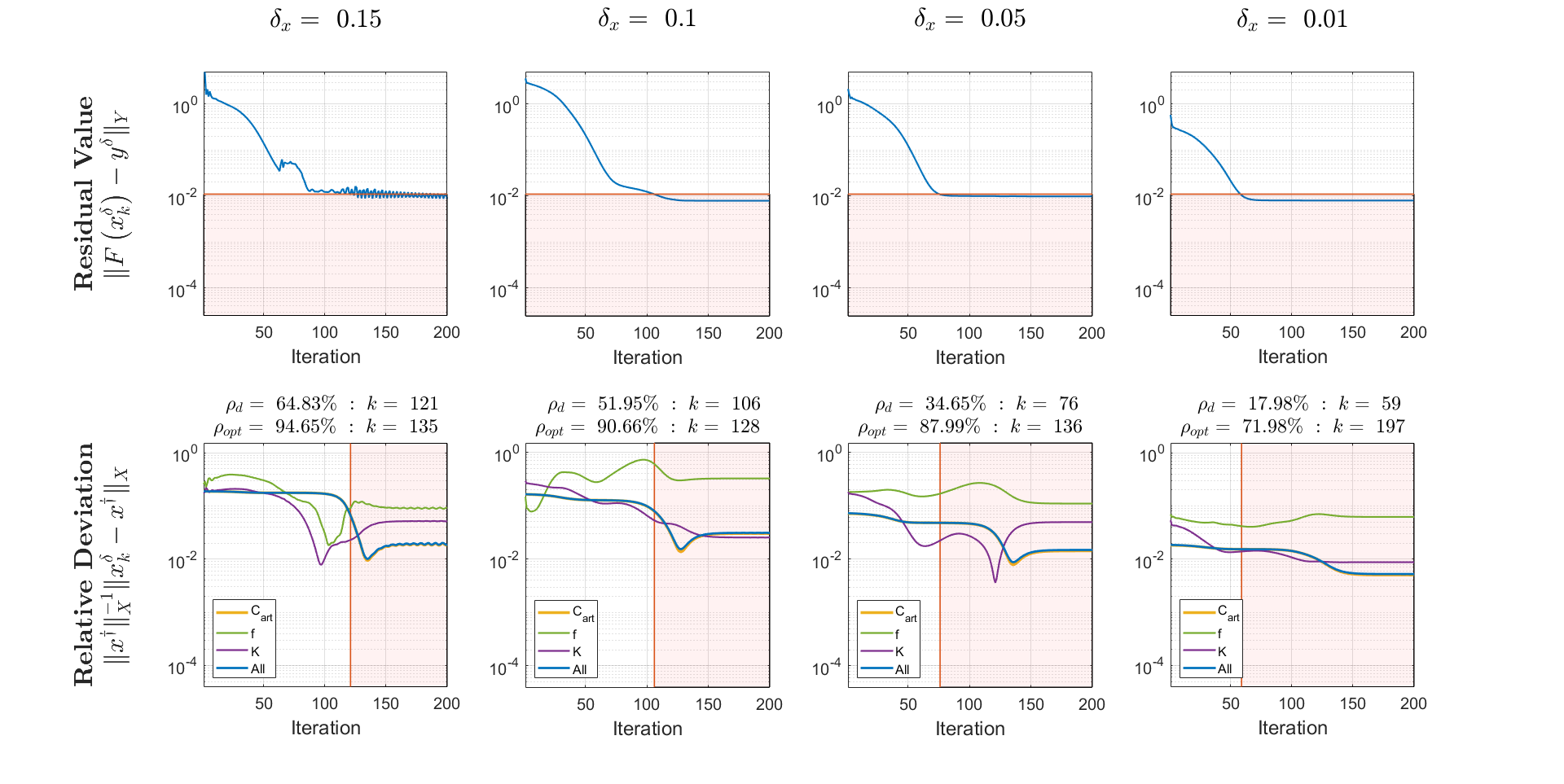}
		\caption[Performance of IRGNM for \(\delta_y =10^{-2}\) for different initial guesses]{Performance of IRGNM for \(\delta_y =10^{-2}\) and different initial guesses}\label{fig:irgnm_del_c2}
	\end{center}
\end{figure}

In a second set of experiments, we consider the situation that not only
$C_\tot$, but also measurements of the values of $C_\art$ are available at the
time-points $t_1,\ldots,t_T$. While it is possible in practice to obtain those
values via blood sample analysis, this procedure is time consuming and
expensive, such that it is a relevant question to what extent such samples
improve the identifiability of the tissue parameters.

Again, for each combination of $\delta_y$ and $\delta_x$, each experiment was
carried out 100 times, the number of divergent experiments is shown in Table
\ref{tab:irgnm} and the Figures show the experiment whose performance was
closest to the median performance of the non-divergent experiments.

Results are shown in Figures \ref{fig:noiseless} to \ref{fig:noise_dy0c2},
where the quantities shown in and above the plots are the same as in Figures
\ref{fig:irgnm_del_0} to \ref{fig:irgnm_del_c2}. It can be observed that the
performance with known $C_\art$ is improved compared to the situation where
only $C_\tot$ is known, across all choices of $\delta_y$ and $\delta_x$. While
for $\delta_y \in \{0,10^{-4}\}$, both versions yield acceptable results, for
$\delta_y \in \{10^{-3},10^{-2}\}$ knowledge of $C_\art$ enables a good
approximation of the ground truth parameters in situations where this was not
possible with knowing only $C_\tot$. This indicates that, as one would expect,
the problem of also identifying $f$ from measurements is significantly more
difficult and there is a benefit (at least for this solution methods) in
measuring $C_\art$ via blood samples.

\begin{figure}[h!]
	\begin{center}
		\includegraphics[scale=0.25, trim=2cm 0 4cm 0, clip]{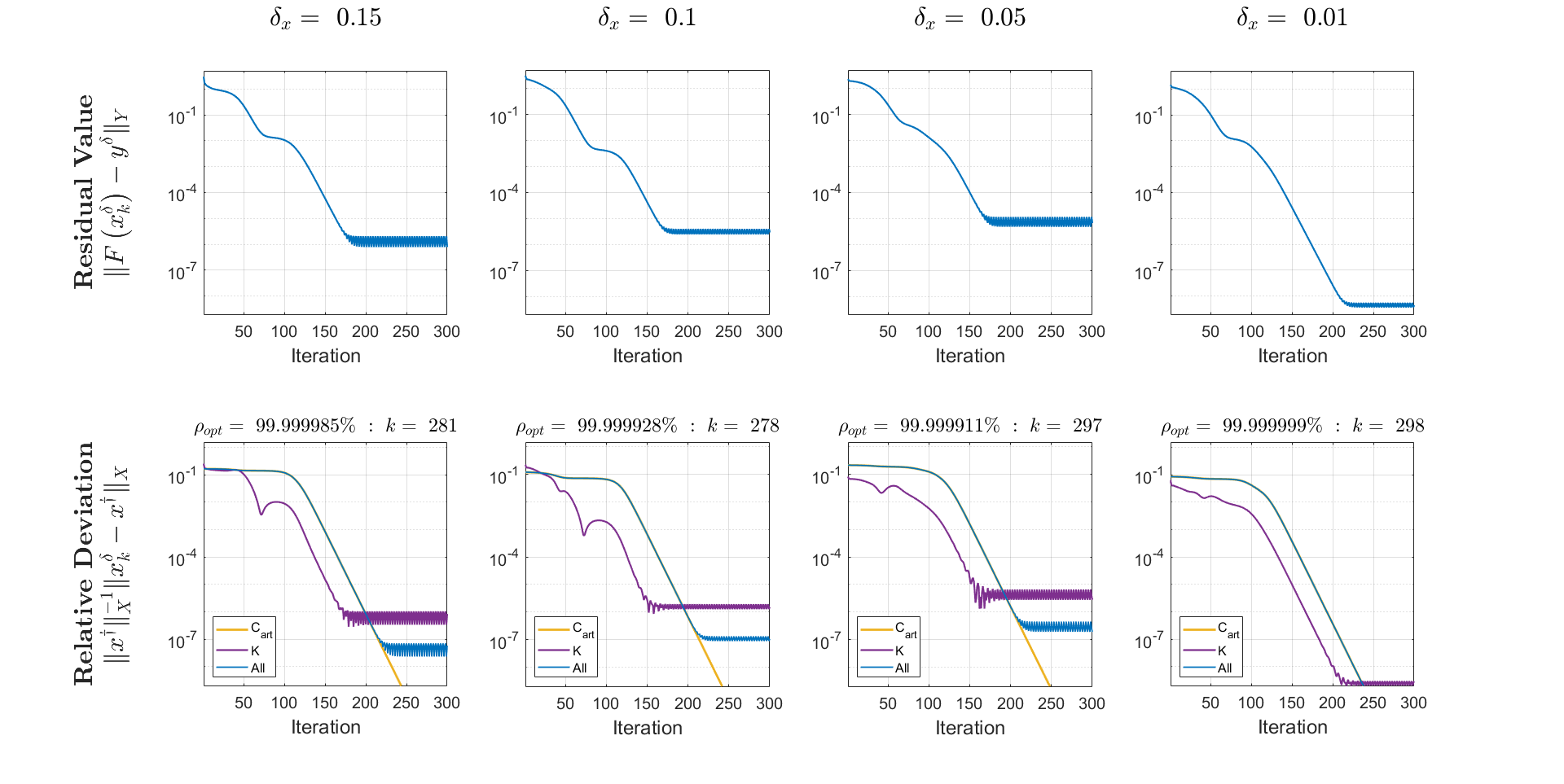}
		\caption[Performance of IRGNM for \(\delta_y =0\) for different initial guesses and known parent plasma fraction $f$]{Performance of IRGNM for \(\delta_y =0\) and different initial guesses and known parent plasma fraction $f$}\label{fig:noiseless}
	\end{center}
\end{figure}

\begin{figure}[h!]
	\begin{center}
		\includegraphics[scale=0.25, trim=2cm 0 4cm 0, clip]{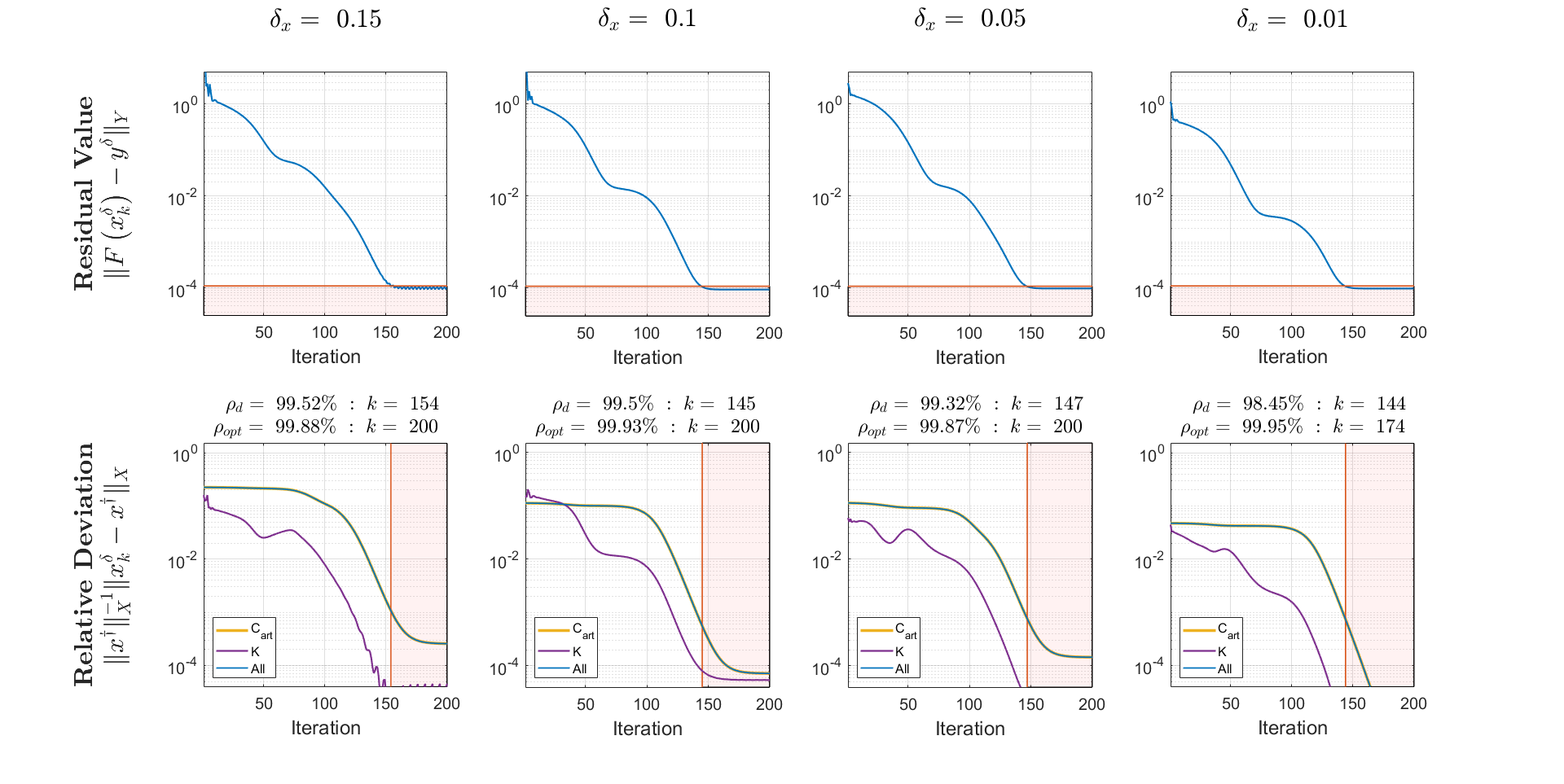}
		\caption[Performance of IRGNM for \(\delta_y =10^{-4}\) for different initial guesses and reduced operator]{Performance of IRGNM for \(\delta_y =10^{-4}\) and different initial guesses and reduced operator}\label{fig:noise_dy0c4}
	\end{center}
\end{figure}
\begin{figure}[h!]
	\begin{center}
		\includegraphics[scale=0.25, trim=2cm 0 4cm 0, clip]{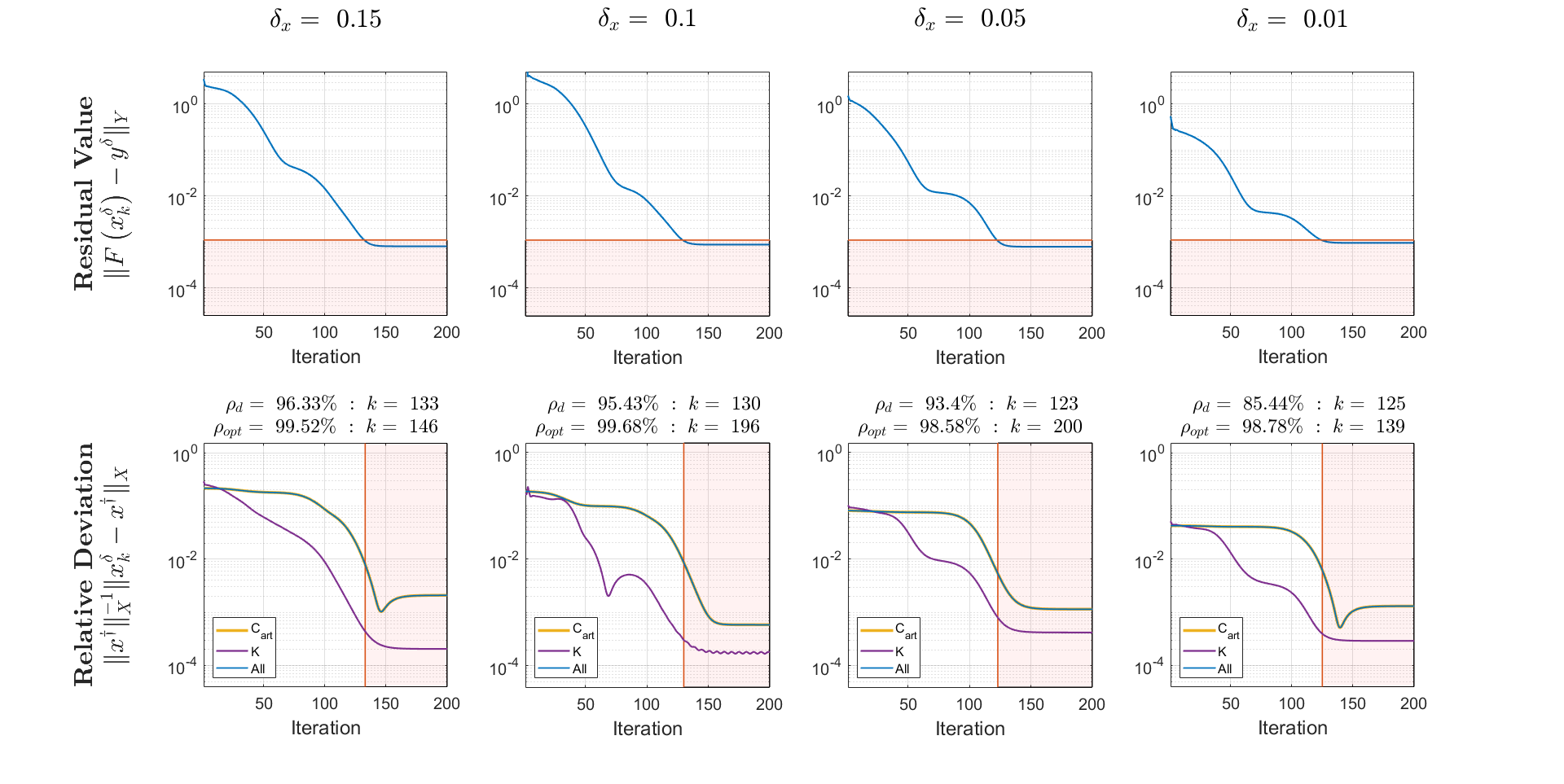}
		\caption[Performance of IRGNM for \(\delta_y =10^{-3}\) for different initial guesses and reduced operator]{Performance of IRGNM for \(\delta_y =10^{-3}\) and different initial guesses and reduced operator}\label{fig:noise_dy0c3}
	\end{center}
\end{figure}

\begin{figure}[h!]
	\begin{center}
		\includegraphics[scale=0.25, trim=2cm 0 4cm 0, clip]{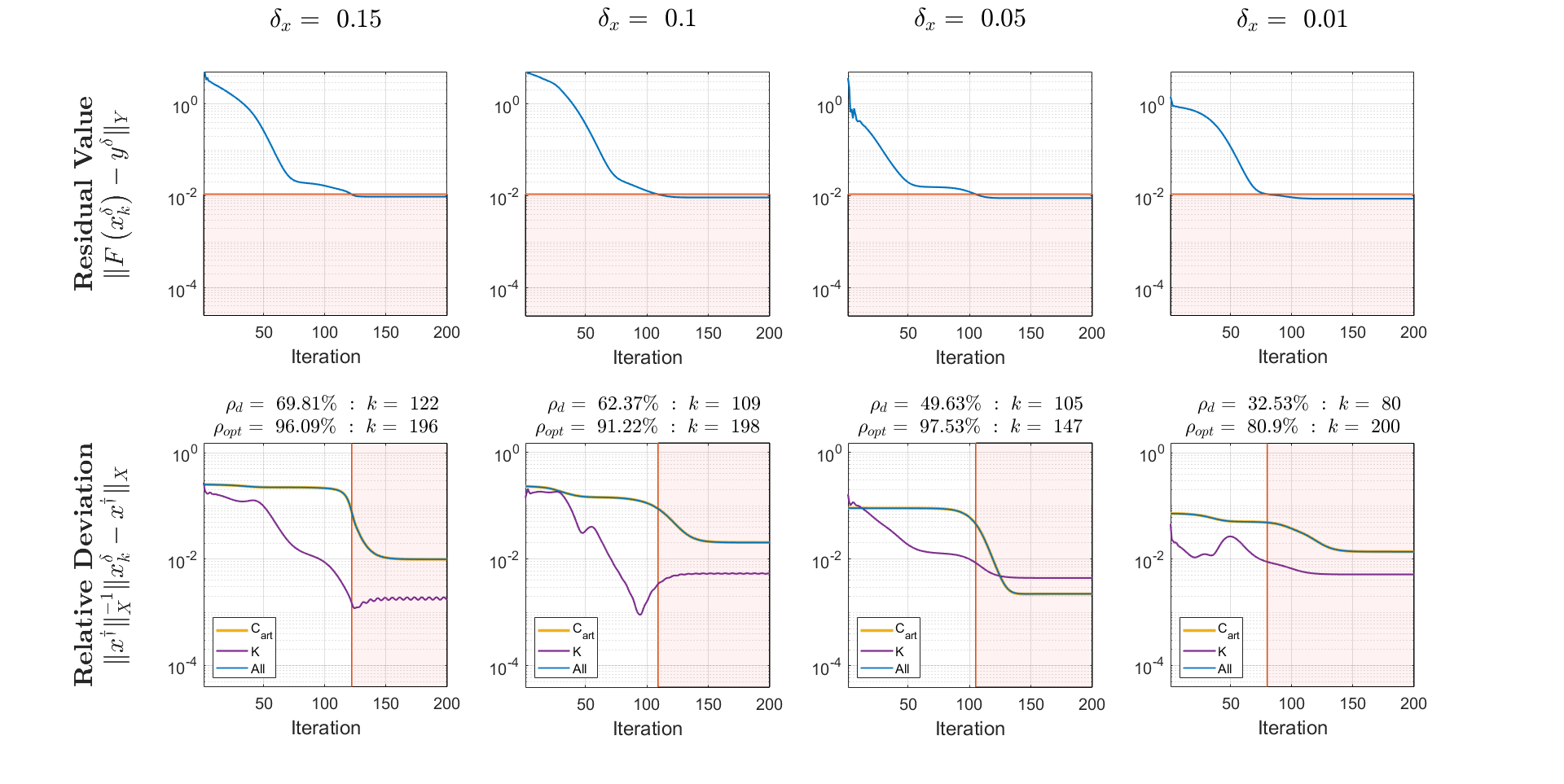}
		\caption[Performance of IRGNM for \(\delta_y =10^{-2}\) for different initial guesses and reduced operator]{Performance of IRGNM for \(\delta_y =10^{-2}\) and different initial guesses and reduced operator}\label{fig:noise_dy0c2}
	\end{center}
\end{figure}

\section{Conclusions}
\label{sec:conclusions}

In this work, we have shown that most tissue parameters of the irreversible
two-tissue compartment model in quantitative PET imaging can, in principle, be
recovered from standard PET measurements only. Furthermore, a full recovery of
all parameters is possible provided that sufficiently many measurements of the
total arterial concentration are available. This is important, since it shows
that parameter recovery is possible via using only quantities that are easily
obtainable in practice, either directly from the acquired PET images or with a
relatively simple analysis of blood samples. While these results consider the
idealized scenario of noiseless measurements, we have further shown that
standard Tikhonov regularization applied to this setting yields a stable
solution method that is capable of exact parameter identification in the
vanishing noise limit.

These findings open the door to a comprehensive numerical investigation of
parameter identification based on only PET measurements and estimates of the
total arterial tracer concentration, using real measurement data and advanced
numerical algorithms. While the numerical results in this paper provide a first
indication that it can be possible to transfer our analytical results to
concrete applications, we expect that a comprehensive effort is necessary to
obtain a practically usable numerical framework for parameter identification.
This will be the next step of our research in that direction.

\clearpage
\bibliographystyle{siamplain}
\bibliography{references}

\end{document}